\documentclass[10pt]{amsart}
\usepackage{amsmath,amsthm,amstext}
\usepackage{amsfonts,amssymb}
\usepackage{mathtools}
\usepackage[mathscr]{eucal}
\usepackage[]{hyperref}
\usepackage{setspace,url}
\usepackage{tikz}
\usepackage{caption}
\usepackage{subcaption}

\newtheorem{lemma}{Lemma}[section]

\newtheorem{proposition}{Proposition}[section]
\newtheorem{remark}{Remark}[section]

\numberwithin{equation}{section}

\renewcommand{\H}{\mathcal{H}}

\newcommand{\D}{\mathcal{D}}
\renewcommand{\P}{\mathcal{P}}

\newcommand{\calO}{\ensuremath{\mathcal{O}}} 

\newcommand{\sgn}[1]{{\rm sgn} (#1)}

{\begin{trivlist} \item[]{\textbf{Proof} }}%
{\hspace*{\fill}$\rule{.3\baselineskip}{.35\baselineskip}$\end{trivlist}}

\begin{document}

\title[A Hamiltonian Dysthe equation for hydroelastic waves]
{A Hamiltonian Dysthe equation for hydroelastic waves in a compressed ice sheet}
\author{Philippe Guyenne}
\author{Adilbek Kairzhan}
\author{Catherine Sulem}
\address[Guyenne]{Department of Mathematical Sciences, University of Delaware, Newark, DE 19716, USA.} \email{guyenne@udel.edu}
\address[Kairzhan]{Department of Mathematics, Nazarbayev University, 010000, Astana, Kazakhstan.} 
\email{akairzhan@nu.edu.kz}
\address[Sulem]{Department of Mathematics, University of Toronto, Toronto, Ontario M5S 2E4, Canada} 
\email{sulem@math.utoronto.ca}





\begin{abstract}
Nonlinear hydroelastic waves along a compressed ice sheet lying on top of a two-dimensional fluid of infinite depth are investigated.
Based on a Hamiltonian formulation of this problem and by applying techniques from Hamiltonian perturbation theory,
a Hamiltonian Dysthe equation is derived for the slowly varying envelope of modulated wavetrains.
This derivation is further complicated here by the presence of cubic resonances for which a detailed analysis is given.
A Birkhoff normal form transformation is introduced to eliminate non-resonant triads while accommodating resonant ones.
It also provides a non-perturbative scheme to reconstruct the ice-sheet deformation from the wave envelope.
Linear predictions on the modulational instability of Stokes waves in sea ice are established,
and implications for the existence of solitary wavepackets are discussed for a range of values of ice compression relative to ice bending.
This Dysthe equation is solved numerically to test these predictions.
Its numerical solutions are compared to direct simulations of the full Euler system, and very good agreement is observed.
\end{abstract}

\maketitle


\section{Introduction}
This paper is devoted to the study of hydroelastic waves describing the deformations of an ice sheet floating on water,
a problem of importance in the polar regions.
It is part of a large class of hydroelasticity problems concerning the interaction between deformable bodies and a surrounding fluid,
and it has many engineering and industrial applications. 
A major challenge concerns modeling the mutual interactions between sea ice and water waves.
On one hand, the presence of sea ice affects the wave dynamics in various different ways.
On the other hand, waves can deform the sea ice, move it vertically and horizontally, and possibly break it up.
There is a large literature on linear models, valid for small-amplitude waves and small ice deflections,
which have focused on quantifying wave attenuation due to sea ice via scattering or other dissipative processes \cite{CGG19,S20}.
In this framework, a multitude of configurations have been considered, including a continuous ice cover
or a fragmented ice cover made of multiple discrete floes with possibly different characteristics.

However, intense wave-in-ice events have been reported and their analysis suggests that linear theory is not sufficient
to explain these observations \cite{M03}.
It is also expected that, in the context of climate change, the proliferation of open water (or more compliant sea ice)
in the polar regions will promote wave growth with larger amplitudes and stronger nonlinear effects.
Motivated in part by these considerations, nonlinear theory has drawn increasing attention in recent years,
with an emphasis on describing nonlinear ice deformations subject to water wave excitation.
Such studies have typically represented the ice cover as a continuous elastic plate of infinite extent,
with various effects depending on the complexity of the elasticity model. 
Their results range from direct numerical simulations of the full nonlinear equations
to weakly nonlinear predictions in asymptotic scaling regimes of interest,
for freely evolving wave solutions or generated by a load moving on the ice \cite{DKP19}.

We are interested here in the modulational regime where approximate solutions are sought 
in the form of slow modulations of near-monochromatic waves.
In this case, perturbation calculations typically yield the nonlinear Schr\"odinger (NLS) equation 
which governs the nonlinear dispersive behavior of the slowly varying wave envelope at leading order.
The NLS equation is a canonical reduced model that arises in many areas of nonlinear science.
Aside from its relative simplicity, its popularity owes much to its many mathematical properties,
including a Hamiltonian formulation and exact traveling wave solutions such as envelope solitons.

Recent findings from both field observations and direct numerical simulations support the relevance of this modulational regime
for wave-ice interactions in the ocean \cite{XG23}.
In particular, several sets of measurements from the Arctic Ocean have found that groupiness is a common trait
of the wave field under open-water and ice-covered conditions \cite{GMT21}.
These observations even suggest that the group structure is enhanced when waves interact with sea ice.
Wave groups with their inherent nonlinearity may be prone to wave focusing due to modulational instability
and may produce unusually large amplitudes \cite{CRMB15}, 
which has implications for ice breaking and ice decline as well as for safety of ship navigation in the polar regions.
Furthermore, this modulational approximation is of interest to the related problem where hydroelastic waves
are induced by a moving load, having in mind e.g. frozen lakes that are used in winter for roads and aircraft runways.
Experiments in Lake Saroma by \cite{T85} showed that a ski-doo snowmobile traveling at a certain speed 
generates localized disturbances of appreciable amplitude, which bear resemblance to solitary wavepackets or envelope solitons.

For this hydroelastic problem, a body of work using modulation theory has developed NLS models in a variety of configurations.
Pioneering results have been obtained by \cite{LM88} and \cite{MS92},
with the former authors considering nonlinear effects from ice bending, compression and inertia
based on Kirchhoff--Love (KL) formulation of a thin elastic plate, while the latter authors incorporated linearized versions of these effects (without inertia).
Nonlinear models based on KL theory have been widely used in both mathematical and numerical investigations.
Subsequent examples include \cite{PD02} who examined the response of a floating ice sheet to a moving load
and derived a steady form of the NLS equation with a forcing term.
They showed that solitary wavepackets exist for certain ranges of parameter values.
\cite{MV-BW11} performed direct numerical simulations of the forced and unforced wave dynamics,
with a focus on solutions having near-minimum phase speeds.
They also produced asymptotic results based on a time-dependent NLS equation that was found to be of defocusing type 
at this minimum (i.e. for small-amplitude forcing) and was corroborated by the numerics.
More recently, \cite{SS22} proposed another NLS equation with a more elaborate elasticity model
and identified in detail the domains of modulational instability for a broad range of physical parameters.
Unlike \cite{MV-BW11} or \cite{PD02}, they adopted an extended version of the KL representation for ice bending
and also included effects from ice compression and inertia.
A similar parametric analysis was conducted by \cite{Hartmann20} using the NLS equation of \cite{LM88}.

Alternatively, an important advance has been made by \cite{PT11} who devised a thin-plate formulation
based on the special Cosserat theory of hyperelastic shells, coupled with nonlinear potential-flow theory for the fluid dynamics.
A distinctive feature of their model as compared to KL is a more nonlinear dependence of the bending force exerted by
the ice cover on the water surface, together with the fact that it has a conservative form.
This allows such a model to be framed within the classical Hamiltonian formulation of the water wave problem by \cite{Z68},
providing a generalization to nonlinear hydroelastic waves.
A number of subsequent studies have adopted Cosserat theory, including \cite{GP12,GP14}
and \cite{MW13} who conducted a weakly nonlinear analysis in the modulational regime
and recovered earlier predictions from KL \cite{MV-BW11} on the non-existence of small-amplitude wavepackets 
(i.e. a defocusing NLS equation at the minimum phase speed).

Beyond the NLS equation in this asymptotic limit, the next-order approximation as originally proposed by \cite{D79}
for surface gravity waves on deep water, has also drawn much attention and has since been extended to other settings.
This so-called Dysthe equation has been shown to compare better with laboratory experiments and direct numerical simulations
than the NLS equation does for larger wave steepnesses.
For example, it can capture the asymmetry of evolving wavepackets, whereas the NLS equation is unable to do so \cite{GKSX21}.
Another effect arising at this higher level of approximation is the wave-induced mean flow 
which has an influence on the stability of finite-amplitude waves.
However, unlike the NLS equation, earlier versions of the Dysthe equation lack a Hamiltonian structure
that would comply with the primitive equations (i.e. the Euler system).
Because it is desirable that important structural properties such as energy conservation be preserved for various reasons,
recent effort has been devoted to establishing a Hamiltonian version of the Dysthe equation.
In particular, \cite{CGS21}, \cite{GKS22-siam} and \cite{GKS22-jfm}
developed a systematic approach to derive Hamiltonian Dysthe equations for water waves with or without a shear current 
by applying techniques from Hamiltonian perturbation theory, which involve reduction to normal form, 
homogenization and other canonical transformations.
Such a reduction is achieved by eliminating non-resonant triads according to the linear dispersion relation.
As a byproduct of this approach, the normal form transformation offers a non-perturbative procedure to reconstruct
the surface elevation from the wave envelope, via the solution of an auxiliary Hamiltonian system of evolution equations.
This contrasts with more standard techniques like the method of multiple scales
where the reconstruction is implemented perturbatively via a Stokes expansion.
For hydroelastic waves, we are not aware of any previous report on a (Hamiltonian or non-Hamiltonian) Dysthe equation from the literature.

Pursuing this line of inquiry, the present work is an extension of \cite{GP12} in several important directions.
The starting point is a Hamiltonian formulation for nonlinear potential flow,
coupled with a nonlinear representation of the ice cover based on Cosserat theory \cite{PT11}. 
Our viewpoint is motivated in part by the significance of Zakharov's Hamiltonian formulation which has formed the theoretical basis
for a countless number of results on nonlinear water waves ranging from rigorous mathematical analysis to operational wave forecasting.
The Dirichlet--Neumann operator is introduced to accomplish the reduction to a lower-dimensional system in terms of surface variables,
and its Taylor series expansion is exploited to carry out the perturbation calculations.
We focus on the two-dimensional problem of hydroelastic waves on water of infinite depth.
In this framework, our new contributions include:

\begin{enumerate}

\item Consideration of nonlinear models for both ice bending and compression,
which is a refinement over previous studies based on simpler elasticity models \cite{GP12,MS92,MV-BW11,MW13}.
The competition between ice bending and compression can affect the focusing/defocusing nature
of this modulational regime, with implications on the existence of solitary wavepackets.

\item Derivation of a Dysthe equation for the wave envelope, with a well-defined Hamiltonian structure and a conserved energy.
This follows from a Birkhoff normal form transformation which is given by the explicit solution of a cohomological relation,
involving an auxiliary system of integro-differential equations in the Fourier space.
An additional difficulty here is the presence of resonant cubic terms due to the more complicated dispersion relation.

\item Development of a splitting scheme in the Fourier space to handle these resonant triads.
Their detailed analysis leads to corrections in the normal form transformation as well as in the envelope equation,
especially for the mean-flow term.
Such corrections are absent in the pure gravity case \cite{CGS21,GKS22-siam} and differ fundamentally from previous adjustments 
to the Dysthe equation in similar situations such as gravity-capillary waves \cite{H85}.

\item Linear analysis of the modulational stability of Stokes waves in sea ice
and validation against numerical solutions of the Dysthe equation under these conditions.
A comparison with NLS predictions and with direct numerical simulations of the full Euler system is also presented.
For this purpose, the surface reconstruction is accomplished in a non-perturbative manner by inverting the normal form transformation.

\end{enumerate}

The remainder of this paper is organized as follows.
The hydroelastic problem under consideration is described in Section 2,
including the Hamiltonian formulation in terms of the Dirichlet--Neumann operator,
the Taylor expansion for small-amplitude waves, and the linear dispersion relation.
Section 3 examines the issue of cubic resonances and proposes a treatment in this theoretical setting.
Section 4 presents the basic tools from Hamiltonian perturbation theory,
focusing on the development of the Birkhhoff normal form transformation 
to deal with resonant and non-resonant triads.
The modulational Ansatz for weakly nonlinear quasi-monochromatic waves is introduced in Section 5,
and the Dysthe equation with its associated Hamiltonian is derived in Section 6.
Finally, Section 7 gives an analytical prediction on modulational instability,
which is tested against numerical solutions of the Dysthe equation.
This model's performance is also assessed relative to NLS computations and direct simulations of the full Euler system.
Results are discussed for a range of ice and wave parameters.

\section{Formulation of the problem} \label{section:formulation}

\subsection{Equations of motion}

We consider nonlinear hydroelastic waves propagating along an elastic sheet (e.g. ocean waves in sea ice) on top of a two-dimensional ideal fluid of infinite depth. In the thin-plate approximation, the ice sheet is assumed to coincide with the fluid surface $\{ y = \eta(x, t)\}$ and to bend in unison with it. 
The fluid domain is then given by $S(\eta) := \{ (x,y): ~x\in \mathbb{R}, ~ -\infty < y <\eta(x,t) \}.$ Assuming that the fluid is incompressible, inviscid and irrotational, it is described by a potential flow such that the velocity field $\mathbf{u}(x,y,t) = \nabla \varphi$ satisfies 
\begin{equation*}
\nabla^2 \varphi = 0 ,
\end{equation*}
in the fluid domain $S(\eta, t)$, where the symbol $\nabla$ denotes the spatial gradient $(\partial_x, \partial_y)$. 

On the surface $y = \eta(x,t)$, we impose the standard kinematic condition  
\begin{equation}
\label{kinematic-condition}
\partial_t \eta = \partial_y \varphi - (\partial_x \eta) (\partial_x \varphi), 
\end{equation}
as well as  the dynamic boundary condition
\begin{equation}
\label{dynamic-conditions}
\partial_t \varphi = -g\eta - \frac{1}{2} \big(\partial_x^2 \varphi + \partial_y^2 \varphi \big) -  
\D \Big( \partial_{s}^2 \kappa + \frac{1}{2} \kappa^3 \Big) - \P \kappa . 
\end{equation}

This condition is derived from the Bernoulli balance of forces  following \cite{PT11} for the interaction between an elastic ice sheet and water beneath it, while the compression term is proportional to the curvature of the fluid-ice interface.
Here, $g$ is the acceleration due to gravity, $\D$ is the coefficient of flexural rigidity, $\P$ is the coefficient of ice compression, $\kappa$ is the curvature of the fluid-ice interface caused by the plate deflection and $s$ is the arclength along this interface. 
The curvature $\kappa$, in terms of $\eta$, is given by
\begin{equation}
\label{curvature}
\kappa = \frac{\partial_x^2 \eta}{\big( 1 + (\partial_x \eta)^2 \big)^{3/2}} ,
\end{equation}
so that the part in \eqref{dynamic-conditions} that describes the bending of the ice sheet becomes 
{\small \begin{equation}
\label{curvature-relation}
\begin{aligned}
\partial_s^2 \kappa + \frac{1}{2} \kappa^3 = \frac{1}{\sqrt{1+(\partial_x \eta)^2}} \partial_x \Bigg( \frac{1}{\sqrt{1+(\partial_x \eta)^2}} \partial_x \Bigg( \frac{\partial_x^2 \eta}{\big( 1 + (\partial_x \eta)^2 \big)^{3/2}} \Bigg) \Bigg) + \frac{1}{2} \Bigg( \frac{\partial_x^2 \eta}{\big( 1 + (\partial_x \eta)^2 \big)^{3/2}} \Bigg)^3 .
\end{aligned}
\end{equation}}

The boundary condition at the bottom reads
\begin{equation}
\label{condition-at-infinity}
|\nabla \varphi| \to 0 \quad \text{as} \quad y \to -\infty .
\end{equation}


\subsection{Hamiltonian formulation}

Following \cite{Z68} and \cite{CS93}, the system of equations \eqref{kinematic-condition}--\eqref{condition-at-infinity} has a Hamiltonian formulation in terms of the variables $(\eta, \xi)$, where $\xi(x,t) = \varphi (x, \eta(x,t), t)$ denotes the trace of the fluid velocity potential evaluated at the surface \cite{GP12}. 

Indeed, introducing the Dirichlet--Neumann operator (DNO) for the fluid domain, which associates to the Dirichlet data $\xi$ on $y=\eta(x,t)$ the normal derivative of $\varphi$ at the surface with a normalizing factor, namely
\begin{equation}
\label{dno-definition}
G(\eta): \xi \mapsto \sqrt{1+|\partial_x \eta|^2} \, \partial_{\text{n}} \varphi \big|_{y=\eta},
\end{equation}
the equations of motion \eqref{kinematic-condition}--\eqref{dynamic-conditions} are equivalent to the following canonical Hamiltonian system
\begin{equation}
\label{Hamiltonian-formulation-spatial}
\partial_t \begin{pmatrix}
\eta \\ \xi
\end{pmatrix} = \begin{pmatrix}
0 & 1 \\ -1 & 0
\end{pmatrix} 
\begin{pmatrix}
\partial_\eta \H \\ \partial_\xi \H
\end{pmatrix}, 
\end{equation}
where the Hamiltonian $\H(\eta, \xi)$ is the total energy 
\begin{equation}
\label{Hamiltonian}
\begin{aligned}
\H (\eta, \xi) = \frac{1}{2} \int_\mathbb{R}  \xi G(\eta) \xi \, dx + \frac{1}{2} \int_\mathbb{R} \Bigg(
g\eta^2 + \D \frac{(\partial_x^2 \eta)^2}{\big( 1+ (\partial_x \eta)^2 \big)^{5/2}} - 2\P \Big( \sqrt{1+(\partial_x \eta)^2} -1 \Big) \Bigg) dx.
\end{aligned}
\end{equation}

The first integral is the kinetic energy, while the second integral is the potential energy due to gravity, bending (or rigidity) and compression.
The Hamiltonian $\H$, together with the momentum 
\begin{equation}
\label{momentum}
I = \int_\mathbb{R} \eta \, (\partial_x \xi) dx,
\end{equation}
and the volume
\begin{equation}
\label{V-invariant}
 V = \int_\mathbb{R} \eta \, dx,
 \end{equation}
 are invariants of motion.

In Fourier variables, the system \eqref{Hamiltonian-formulation-spatial} preserves its canonical form. Indeed, denoting the Fourier transform of $f(x)$ by $\hat{f_k} = (1/\sqrt{2\pi}) \int_\mathbb{R} e^{-ikx} f(x) dx$, we have 
\begin{equation}
\label{Hamiltonian-formulation-fourier}
\partial_t \begin{pmatrix}
\hat \eta_{-k} \\ 
\hat \xi_{-k}
\end{pmatrix} = \begin{pmatrix}
0 & 1 \\ -1 & 0
\end{pmatrix} 
\begin{pmatrix}
\partial_{\hat \eta_k} \H \\ \partial_{\hat \xi_k} \H
\end{pmatrix},
\end{equation}
where we have used that $(\hat \eta_{-k}, \hat \xi_{-k}) = (\overline{\hat \eta}_k, \overline{\hat \xi}_k)$ since $\eta(x)$ and $\xi(x)$ are real-valued functions. Due to the conservation of volume \eqref{V-invariant}, we can choose $\hat \eta_0 = 0.$ In the following, we drop the hats when denoting Fourier modes if there is no confusion.

\subsection{Non-dimensionalization} 

The ice parameters are defined by 
$$
\D = \frac{\sigma}{\rho}, \quad  \sigma = \frac{Eh^3}{12(1-\nu^2)},
 \quad \P = \frac{Ph}{\rho}, 
$$
where $\rho$ is the density of the underlying fluid, $h$ is the thickness of the ice sheet, $\nu$ is Poisson's ratio for ice, $E$ is Young's modulus and $P$ is the compressive stress.
Numerical values of these physical parameters are listed in Table 1 of \cite{PD02} for two sets of experimental data.
Introducing the characteristic length and time scales
\begin{equation*}
\ell = \left(\frac{\sigma}{\rho g} \right)^{1/4} = \left( \frac{\D}{g} \right)^{1/4}, \quad \tau = \left(\frac{\sigma}{\rho g^5} \right)^{1/8} = \left( \frac{\D}{g^5} \right)^{1/8},
\end{equation*}
respectively \cite{GP12,XG23}, the dimensionless equations of motion follow \eqref{Hamiltonian-formulation-spatial} with Hamiltonian \eqref{Hamiltonian} modulo the change
\begin{equation*}
g \to 1, \quad \D \to 1, \quad \P \to \frac{Ph}{\sqrt{\sigma \rho g}} = \frac{\P}{\sqrt{g \D}},
\end{equation*}
where the third dimensionless parameter measures the relative importance of compression to gravity and rigidity. 
In typical physical situations, it is of order 1 (see e.g. \cite{LM88} for realistic values of physical parameters and a discussion by \cite{SS22}).

For convenience, we will use the same notations hereafter but it is understood that all variables and parameters are now dimensionless
according to this choice of non-dimensionalization.


\subsection{Taylor expansion of the Hamiltonian near equilibrium} 

It is known that the DNO is analytic in $\eta$ \cite{CM85}, and admits a convergent Taylor series expansion
\begin{equation} \label{DNOseries}
G(\eta) = \sum_{m=0}^\infty G^{(m)}(\eta),
\end{equation}
about $\eta=0$. For each $m$, the term $G^{(m)}(\eta)$ is homogeneous of degree $m$ in $\eta$, and can be calculated explicitly via recursive relations \cite{CS93}. Denoting $D = -i \, \partial_x$, the first three terms are
\begin{equation}
\label{g-012-recursive}
\left\{ \begin{array}{l}
G^{(0)}(\eta)  = |D|, \\
G^{(1)}(\eta) = D \eta D - G^{(0)}\eta G^{(0)}, \\
G^{(2)}(\eta) = -\frac{1}{2} \left( |D|^2 \eta^2 G^{(0)}+ G^{(0)} \eta^2 |D|^2 - 2G^{(0)} \eta G^{(0)} \eta G^{(0)} \right).
\end{array} \right.
\end{equation}
This in turn  provides an expansion of
the Hamiltonian near the stationary solution $(\eta,\xi)  = (0,0)$:
\begin{equation*}
\begin{aligned}
\H (\eta, \xi) = \frac{1}{2} \int_\mathbb{R} \Big[ & \xi G^{(0)}(\eta) \xi + g\eta^2 + \D (\partial_x^2 \eta)^2 - \P (\partial_x \eta)^2 + \xi G^{(1)}(\eta) \xi\\
& + \xi G^{(2)} (\eta) \xi - \frac{5}{2} \D (\partial_x \eta)^2 (\partial_x^2 \eta)^2 + \frac{1}{4} \P (\partial_x \eta)^4 + \dots
\Big] dx.
\end{aligned}
\end{equation*}
In Fourier variables, the above expansion can be written as
\begin{equation}
\label{HH}
\H = \H^{(2)} + \H^{(3)}+ \H^{(4)} + \dots
\end{equation}
where each term $\H^{(m)}$ is homogeneous of degree $m$ in $\eta$ and $\xi$. In particular, we have
{\small
\begin{equation}
\label{H2-H3-in-fourier}
\begin{aligned}
& \H^{(2)} = \frac{1}{2} \int \Big( |k| |\xi_k|^2 + (g - \P k^2 + \D k^4) |\eta_k|^2 \Big) dk,\\[3pt]
& \H^{(3)} = - \frac{1}{2\sqrt{2\pi}} \int (k_1 k_3 + |k_1| |k_3|) \xi_1 \eta_2 \xi_3 \delta_{123} dk_{123},
\end{aligned}
\end{equation}
}
and 
\begin{equation}
\label{H4-in-fourier}
\begin{aligned}
\H^{(4)} =& ~ - \frac{1}{8\pi} \int |k_1| |k_4| (|k_1| + |k_4| - 2 |k_3+k_4|) \xi_1 \eta_2 \eta_3 \xi_4 \delta_{1234} dk_{1234}\\[3pt]
& ~ + \frac{1}{4\pi} \int \Big( \frac{5 \D}{2} k_1^2 k_2^2 k_3 k_4 {+} \frac{\P}{4} k_1 k_2 k_3 k_4 \Big) \eta_1 \eta_2 \eta_3 \eta_4 \delta_{1234} dk_{1234}, 
\end{aligned}
\end{equation}
where we have used the compact notations
$(\eta_j, \xi_j)= (\eta_{k_j}, \xi_{k_j})$, $dk_{1 \dots n} = dk_1 \dots dk_n$, and $\delta_{1 \dots n} = \delta(k_1+\dots+k_n)$, where $\delta(\cdot)$ is the Dirac distribution $\delta(k) = (1/2\pi) \int_\mathbb{R} e^{-ikx} dx$.
The domain of integration is now omitted from all integrals and is understood to be $\mathbb{R}$ for each $x_j$ or $k_j$.
Notice that the ice parameters $\P$ and $\D$ appear explicitly in the expressions for $\H^{(2)}$ and $\H^{(4)}$, but not in $\H^{(3)}$.

\subsection{  {Dispersion relation} }

The linearized system for \eqref{Hamiltonian-formulation-spatial} around $(\eta, \xi)=(0,0)$ is
\begin{equation}
\label{linearized}
\begin{aligned}
& \partial_t \eta = |D| \xi,\\
& \partial_t \xi = -g\eta - \P \partial_x^2 \eta - \D \partial_x^4 \eta. 
\end{aligned}
\end{equation}
For
\begin{equation}
\label{dispersion-relation}
\omega^2(k) = \omega_k^2 := |k| \left( g - \P k^2 + \D k^4 \right),
\end{equation}
Eq. \eqref{linearized} admits periodic plane-wave solutions 
{
\begin{equation}
\label{modes}
\eta(x,t) 
\propto e^{i(kx-\omega_k t)} \quad \text{and} \quad \xi(x,t) 
\propto e^{i(kx-\omega_k t)}.
\end{equation}
}
Equation \eqref{dispersion-relation} is known as the linear dispersion relation. 
By construction, for a particular choice of $\P, \D$ and $k$ the value of $\omega_k^2$ may be either positive or negative. When $\omega_k^2$ is positive, the solutions in \eqref{modes} are traveling waves since $\omega_k \in \mathbb{R}$. 
On the other hand, when $\omega_k^2<0$, one  has $\omega_k \in i \, \mathbb{R}$, and therefore, the solutions in \eqref{modes} represent either evanescent or growing waves due to the factor $e^{\pm |\omega_k| t}$ in the expressions of $\eta$ and $\xi$.

{For parameters $ \P, \D$ and $k$ with $\omega_k^2>0$,  it is convenient to introduce the complex symplectic coordinates which diagonalize the quadratic Hamiltonian $\H^{(2)}$ in \eqref{H2-H3-in-fourier}. In the Fourier space, these are given by }
\begin{equation}
\label{eta-xi-to-z-mapping}
\begin{pmatrix}
z_k \\ \bar z_{-k}
\end{pmatrix} = \frac{1}{\sqrt 2} \begin{pmatrix}
a_k & i \, a_k^{-1} \\
a_k & -i \, a_k^{-1} 
\end{pmatrix} \begin{pmatrix}
\eta_k \\ \xi_k
\end{pmatrix},
\end{equation}
where $a_k^2 := \omega_k/|k|$. The quadratic term $\H^{(2)}$ becomes
\begin{equation}
\label{H2-fourier-in-z}
\H^{(2)} = \int \omega_k |z_k|^2 dk,
\end{equation}
while the cubic term $\H^{(3)}$ takes the form
\begin{equation}
\label{H3-fourier-in-z}
\H^{(3)} = \frac{1}{8\sqrt{\pi}} \int ( { k}_1 { k}_3 +|{k}_1|{k}_3|) \frac{a_1 a_3}{a_2} (z_1-\bar z_{-1})(z_2+\bar z_{-2})(z_3-\bar z_{-3})  \delta_{123} d k_{123},
\end{equation}
where $z_{\pm j} := z_{\pm { k}_j}$ and $a_j := a_{{ k}_j}$. Moreover, since the mapping \eqref{eta-xi-to-z-mapping} is canonical \cite{GKS22-siam}, the full Hamiltonian system \eqref{Hamiltonian-formulation-fourier} becomes 
\begin{equation}
\label{z-physical-ham-system}
\partial_t \begin{pmatrix}
z_k \\ \bar z_{-k}
\end{pmatrix} = \begin{pmatrix}
0 & -i \\ i & 0
\end{pmatrix} \begin{pmatrix}
\partial_{z_k} \H \\ \partial_{\bar z_{-k}} \H
\end{pmatrix}.
\end{equation}
We note that the expressions \eqref{H2-fourier-in-z} and \eqref{H3-fourier-in-z} coincide in structure with the corresponding formulas for the quadratic and cubic Hamiltonians in the case of deep-water surface waves \cite{CS16}. This is because the parameters of ice rigidity and ice compression, appearing in our problem, are hidden in the definitions of  $\omega_k$ and $a_k$.

One finds the values of $\P$ when $\omega_k^2$ is positive for any $k$ by analyzing the quartic inequality
\begin{equation*}
\D k^4 - \P k^2 + g > 0. 
\end{equation*}
Recall that $g = 1$ and $\D = 1$ in dimensionless units.
For a given $\P$, one has $\omega_k^2 > 0$ for  small or large $k$. When $k$ takes moderate values,  $\omega_k^2$ could be negative. One can find the zeros of $\omega_k^2$ analytically but it is more convenient to plot the dispersion relation  $\omega_k^2$ as a function of $k$ for various values of the parameter $\P$ 
as shown in figure \ref{figure-omega2-k}. 
Note that $\omega_k^2$ remains always positive whenever $\P < 2$.  When $\P = 2$, $\omega_k^2 = 0$ for $k = 1$. 

We are interested in having  $\omega_k^2 > 0$ for all real values of $k$ since we are looking for traveling wave solutions, not evanescent or growing waves. Thus,  throughout the paper, we make the assumption 
\begin{equation}
\label{B-assumption}
0 \le \P < 2. 
\end{equation}

\section{Resonant triads}\label{section:resonant-triads}

\subsection{Existence of resonant triads}

We now address the question of existence of nonzero \textit{resonant triads} $(k_1, k_2,k_3)$ satisfying 
\begin{equation}
\label{resonance-condition-sum-of-k}
k_1 + k_2 + k_3 = 0, 
\end{equation}
and at least one of the equations 
\begin{equation}
\label{resonance-condition}
\omega_1 \pm \omega_2 \pm \omega_3 = 0,
\end{equation}
where $\omega_j$ stands for $\omega_{k_j}$.
To find all the possible triads satisfying \eqref{resonance-condition-sum-of-k}--\eqref{resonance-condition}, we consider the following function
\begin{equation}
\label{d123-def}
\begin{aligned}
d_{123} = d(k_1,k_2,k_3) := (\omega_1 + \omega_2 + \omega_3)(\omega_1 + \omega_2 - \omega_3)(\omega_1 - \omega_2 + \omega_3)(\omega_1 - \omega_2 - \omega_3).
\end{aligned}
\end{equation}
From the construction of $d_{123}$, we have that $(k_1,k_2,k_3)$ is a resonant triad if and only if it solves the system 
\begin{equation}
\label{d123=0}
\left\{ 
\begin{array}{l}
k_1+k_2+k_3 = 0, \\
d_{123}(k_1,k_2,k_3) = 0, \\
k_j \neq 0, \quad j = \{ 1,2,3 \}.
\end{array}
\right.
\end{equation}
The analogue of the system \eqref{d123=0} was studied in detail by \cite{CS16} for surface gravity water waves. In that case, due to the absence of physical parameters $\P$ and $\D$, the expression for $d_{123}$ is much simpler and the computations can be performed by hand. 
However, in the present case, it is difficult to solve \eqref{d123=0} explicitly due to the complexity of the dispersion relation \eqref{dispersion-relation} and the high degree of the polynomial function $d_{123}$. 

Furthermore, as shown in \cite{CS16}, the function $d_{123}$ appears together with the constraint 
\begin{equation}
\label{constraint-for-d123}
k_1 k_3+|k_1||k_3| \neq 0.
\end{equation}
Therefore, we provide 
estimates for $d_{123}$ under the constraint \eqref{constraint-for-d123} and the assumption $k_1+k_2+k_3=0$. 

\begin{lemma}
\label{lemma-d123}
Assuming $\sgn{k_1} = \sgn{k_3}$ and $k_1+k_2+k_3=0$, we have 
\begin{equation}
\label{d123-tilded123-relation}
\begin{aligned}
d_{123} = k_1k_3 \widetilde d(k_1, k_3),
\end{aligned}
\end{equation}
where 
\begin{equation}
\label{tilde-d123}
\begin{aligned}
\widetilde d (k_1, k_3) := & ~ k_1 k_3 (k_1 + k_3)^2 \left( 3\P -5 \D (k_1^2+k_1k_3+k_3^2) \right)^2\\
& ~ - 4 (g-\P k_1^2 + \D k_1^4)(g-\P k_3^2 + \D k_3^4). 
\end{aligned}
\end{equation}
The function $\widetilde d$ has the following symmetry properties: 
\begin{equation}
\label{tilde-d123-symmetry}
\widetilde d (k_1, k_3) = \widetilde d (k_3, k_1) \quad \text{and} \quad \widetilde d (k_1, k_3) = \widetilde d (-k_1, -k_3). 
\end{equation}
\end{lemma}

\begin{proof}
The results are based on direct computations under the above assumptions. 
Below, we provide a short overview of steps. 

First, rewrite $d_{123}$ as 
\begin{equation}
\label{d123-alternative}
d_{123} = (\omega_1^2 + \omega_3^2 - \omega_2^2)^2 - 4 \omega_1^2 \omega_3^2.
\end{equation}
Observe that the second term on the right-hand side of \eqref{d123-alternative}  identifies with the second line in \eqref{tilde-d123}. For the first term on the right-hand side of \eqref{d123-alternative}, we write 
\begin{equation*}
\begin{aligned}
\omega_1^2 + \omega_3^2 - \omega_2^2 = g (|k_1|+ |k_3|-|k_2|) - \P (|k_1|^3+ |k_3|^3-|k_2|^3) + \D (|k_1|^5+ |k_3|^5-|k_2|^5). 
\end{aligned}
\end{equation*}
Given the assumptions, the first term vanishes since 
\begin{equation}
\label{linear-terms-kj}
|k_1|+ |k_3|-|k_2| = 0,
\end{equation}
and for the cubic terms, 
we have
\begin{equation*}
\label{cubic-terms}
\begin{aligned}
|k_1|^3+ |k_3|^3-|k_2|^3 = |k_2| (k_1^2 - k_1k_3 + k_3^2) - |k_2|^3 = -3 |k_2| k_1k_3. 
\end{aligned}
\end{equation*}
Similarly, for fifth-power terms, we have
\begin{equation*}
\label{fifth-terms}
\begin{aligned}
|k_1|^5+ |k_3|^5-|k_2|^5 = -5|k_2| k_1k_3 (k_1^2 + k_1k_3 + k_3^2).  
\end{aligned}
\end{equation*}
Substituting these expressions back into \eqref{d123-alternative}, we obtain \eqref{d123-tilded123-relation}.
The properties \eqref{tilde-d123-symmetry} are straightforward to verify.
\end{proof}

Based on Lemma \ref{lemma-d123}, finding resonant triads $(k_1, k_2, k_3)$ in \eqref{d123=0} under the constraint \eqref{constraint-for-d123} is equivalent to  finding the nonzero roots of the equation 
\begin{equation}
\label{tilde-d13=0}
\widetilde d (k_1, k_3) = 0 \quad \text{with} \quad \sgn{k_1} = \sgn{k_3},
\end{equation}
where $k_2 = -k_1-k_3$.  
Solving \eqref{tilde-d13=0} explicitly seems to be impossible due to the complicated structure of $\widetilde d(k_1,k_3)$ in \eqref{tilde-d123}. 
Therefore, we find solutions to \eqref{tilde-d13=0} numerically. 
Since $\widetilde d$ is invariant under sign change of variables, $\widetilde d (k_1, k_3) = \widetilde d (-k_1, -k_3)$, we  only concentrate on the positive roots of \eqref{tilde-d13=0} and the negative roots are found by symmetry.  
Setting $\P = 1$ for simplicity,  the solution curve $\mathcal{C}^+$ for positive roots of \eqref{tilde-d13=0} is displayed in figure \ref{figure-resonant-triads}.  A similar solution curve is present for other values of $\P < 2$. 
We note that the curve $\mathcal{C}^+$ is unbounded, has the  $k_1$- and $k_3$-axes as  asymptotes and never intersects these axes. The negative roots of \eqref{tilde-d13=0} are given by the solution curve $\mathcal{C}^-$, which is obtained from $\mathcal{C}^+$ by symmetry with respect to the origin. Combining these curves,
all nonzero roots of \eqref{tilde-d13=0} are given by 
\begin{equation}
\label{C-defn}
\mathcal{C} := \mathcal{C}^+ \cup \mathcal{C}^-. 
\end{equation}

In our analysis later, various integrals involve expressions where the function $\widetilde d(k_1, k_3)$ appears in the denominator. To avoid this ``small divisors'' issue, we will restrict the integral to regions away from the solution curve $\mathcal{C}$. 
To do so, we construct a neighborhood of $\mathcal{C}$, referred to as $\mathcal{C}_\mu$, for sufficiently small $\mu$, such that   
\begin{equation}
\label{d-tilde-bounded-away-from-0}
\begin{aligned}
|\widetilde d(x, y)| \geq \text{const}_\mu \quad \text{for} \quad \text{all } (x, y) \in \mathbb{R}^2 \backslash \mathcal{C}_\mu, 
\end{aligned}
\end{equation}
where $\text{const}_\mu$ is a constant depending on $\mu$ only. 
More precisely, at every point $(k_1, k_3) \in \mathcal{C}^+$, denote a normal vector to $\mathcal{C}^+$ by $\vec{n} (k_1, k_3)$. Then, consider the set of points in the $(k_1,k_3)$-plane given by 
\begin{equation}
\label{tube-C-plus}
\begin{aligned}
\left\{ (k_1, k_3) + \mu \, t \, \vec{n} (k_1,k_3): t \in [-1,1] \text{ and } (k_1, k_3) \in \mathcal{C}^+ \right\},
\end{aligned}
\end{equation}
which geometrically looks like a ``tube'' around the curve $\mathcal{C}^+$ with a width $\mu$. We then construct a neighborhood  of $\mathcal{C}^+$, denoted by $\mathcal{C}_\mu^+$, as 
\begin{equation}
\mathcal{C}_\mu^+ := \{ \text{the set } \eqref{tube-C-plus}\} \cap \{(x,y): x\geq 0, y\geq 0\}. 
\end{equation}
A sketch of $\mathcal{C}_\mu^+$ is shown in figure \ref{figure-resonant-triads} for $\P = 1$. The neighborhood $\mathcal{C}_\mu^-$ for the curve $\mathcal{C}^-$ is constructed similarly, and can be seen as the reflection of $\mathcal{C}_\mu^+$ with respect to the origin. Combining these sets, we get the neighborhood of the set $\mathcal{C}$ in \eqref{C-defn} as 
$\mathcal{C}_\mu := \mathcal{C}_\mu^+ \cup \mathcal{C}_\mu^-$. 

We now provide formal arguments to show that \eqref{d-tilde-bounded-away-from-0} is satisfied for the set $\mathcal{C}_\mu$ constructed above. 
It suffices to show that  \eqref{d-tilde-bounded-away-from-0} is satisfied on the boundary of $\mathcal{C}_\mu^+$ for small values of $k_1$, which corresponds to the far end of the neighborhood in figure \ref{figure-resonant-triads} located at the top left corner. Then, the validity of \eqref{d-tilde-bounded-away-from-0} for all $(x, y)$ outside $\mathcal{C}_\mu$ follows by continuity arguments and symmetry.

First, we note that, under  assumption \eqref{B-assumption}, the points $(k_1, k_3) \in \mathcal{C}$ satisfy
$k_3 \approx (4g/(25 \D k_1))^{1/3}$ whenever $k_1 \ll 1$. Indeed, from \eqref{tilde-d123}, it is clear that if $k_1 \ll 1$ then $k_3 \gg 1$. As a result, for $k_1 \ll 1$, we have 
\begin{equation}
\label{k_1-k_3-asymptotics}
0 = \widetilde d(k_1, k_3) \approx 25 \mathcal{D}^2 k_1 k_3^7 - 4 g \mathcal{D} k_3^4 \implies k_3 \approx \left(\frac{4g}{25 \mathcal{D} k_1} \right)^{1/3}. 
\end{equation}
Therefore, the normal vector of $\mathcal{C}^+$ for small values of $k_1$ is approximately given by 
\begin{equation*}
\vec{n} (k_1,k_3) = \frac{\left( k_1, 3k_1^{4/3} \left( \frac{25\D}{4g} \right)^{1/3} \right)}{\left\| k_1, 3k_1^{4/3} \left( \frac{25\D}{4g} \right)^{1/3} \right\|} \approx (1, 0). 
\end{equation*}
The points on the boundary of $\mathcal{C}_\mu^+$ with $k_1 \ll 1$ are  
\begin{equation*}
(0, k_3) \quad \text{and} \quad (k_1+ \mu, k_3),
\end{equation*}
where $k_3$ satisfies the estimate \eqref{k_1-k_3-asymptotics}. Applying simple algebraic steps to \eqref{tilde-d123}, it can be shown that the value of $\widetilde d (0, k_3)$ under the assumption \eqref{B-assumption} is bounded away from zero by 
\begin{equation*}
\left| \widetilde d (0, k_3) \right| \geq \left| \widetilde d (0, \sqrt{\P/(2\D)}) \right| = \frac{g}{\D} (4g\D - \P^2) >0.
\end{equation*}
To estimate the value of $\widetilde d(k_1 + \mu, k_3)$ we use the Taylor expansion 
\begin{equation}
\label{formal-expansion-tilde-d123}
\begin{aligned}
\widetilde d(k_1 + \mu, k_3) \approx \widetilde d(k_1, k_3) + \mu \, \partial_{k_1} \widetilde d(k_1, k_3) = \mu \, \partial_{k_1} \widetilde d(k_1, k_3). 
\end{aligned}
\end{equation}
Taking the derivative of \eqref{tilde-d123} with respect to $k_1$, we have 
\begin{equation*}
\begin{aligned}
\partial_{k_1} \widetilde d (k_1, k_3) = &~k_3 (3k_1+k_3) (k_1 + k_3) \left( 3\P -5 \D (k_1^2+k_1k_3+k_3^2) \right)^2 \\
& + 2 k_1 k_3 (k_1 + k_3)^2 \left( 3\P -5 \D (k_1^2+k_1k_3+k_3^2) \right) \left( -10 \D k_1 - 5\D k_3 \right)\\
& + 4 (2 \P k_1 - 4 \D k_1^3)(g-\P k_3^2 + \D k_3^4),
\end{aligned}
\end{equation*}
which, under \eqref{k_1-k_3-asymptotics}, is approximated by
\begin{equation*}
\begin{aligned}
\partial_{k_1} \widetilde d (k_1, k_3) \approx 25 \D^2 k_3^7,
\end{aligned}
\end{equation*}
whenever $k_1 \ll 1$. As a result, the estimate in \eqref{formal-expansion-tilde-d123} takes the form 
\begin{equation*}
\widetilde d(k_1 + \mu, k_3) \approx 25 \mu \D^2 k_3^7,
\end{equation*}
which is bounded away from zero as desired. 
Later in the analysis, the function $d_{123} $ will appear in the denominator and we will set $\mu = \varepsilon $, where $\varepsilon$ is the small parameter in the modulational regime.

In view of \eqref{d123=0} and \eqref{d123-tilded123-relation}, the set of resonant triads satisfying \eqref{resonance-condition-sum-of-k}--\eqref{resonance-condition} is given by
\begin{equation}
\label{resonance-set}
\mathcal{R} := \{(k_1, -k_1-k_3, k_3) \in \mathbb{R}^3: (k_1, k_3) \in \mathcal{C} \}.
\end{equation}
The neighborhood of $\mathcal{R}$ is defined by
\begin{equation}
\label{R-neighborhood}
\mathcal{N}_\mu := \{ (k_1, -k_1-k_3, k_3) \in \mathbb{R}^3: (k_1, k_3) \in \mathcal{C}_\mu\},
\end{equation}
as illustrated in figure \ref{figure-resonant-triads}. 
We also define a characteristic function $\chi_{\mathcal{N}_\mu} (k_1, k_2, k_3)$ supported in the neighborhood $\mathcal{N}_\mu$ as
\begin{equation}
\label{characteristic-fnct-defn}
\chi_{\mathcal{N}_\mu} (k_1, k_2, k_3) :=
\left\{ \begin{array}{l}
1, \quad \text{if } (k_1, k_2, k_3) \in \mathcal{N}_\mu, \\
0, \quad \text{otherwise}.
\end{array} \right.
\end{equation}

\subsection{Phase and group speeds}

Assuming $k > 0$ without loss of generality (since $\omega_k$ is an even function of $k$), 
the phase speed associated with the linear dispersion relation \eqref{dispersion-relation} is given by
\[
c(k) = \frac{\omega_k}{k} = \sqrt{\frac{g - {\P} k^2 + {\D} k^4}{k}},
\]
while the group speed reads
\[
c_g(k) = \partial_k \omega_k = \frac{g - 3 {\P} k^2 + 5 {\D} k^4}{2 \omega_k}.
\]
It can be easily shown that, if the phase speed $c$ has a local minimum at $k = k_{\min}$, then
the phase and group speeds coincide at this minimum.
The equation $c = c_g$ reduces to $g + {\P} k^2 - 3 {\D} k^4 = 0$ which is quadratic in $k^2$.
The only possible solution that is real and positive takes the form
\begin{equation} \label{kmin}
k_{\min} = \sqrt{\frac{{\P} + \sqrt{{\P}^2 + 12 g {\D}}}{6 {\D}}}.
\end{equation}
In the flexural-gravity case (${\P} = 0$), this solution yields the well-known value
$k_{\min} = (g/(3 {\D}))^{1/4}$ \cite{GP12}.
Figure \ref{coeff-graph} shows the phase and group speeds for values $\P = \{ 1, 2, 5 \}$.

\section{Transformation theory} \label{section:transf-theory}

The method of Hamiltonian transformation theory has been previously applied to deep-water irrotational gravity waves in two and three dimensions \cite{CGS21,GKS22-siam}, and waves with constant vorticity \cite{GKS22-jfm} to derive a Hamiltonian Dysthe equation for the envelope of the surface elevation. 
In all of these cases, the set of resonant triads, satisfying the analogue of \eqref{resonance-condition}, is either empty or naturally ruled out from the analysis. This is not the case in the present problem.

We recall that the Poisson bracket of two functionals $K(\eta, \xi)$ and $H(\eta, \xi)$ 
of real-valued functions $\eta$ and $\xi$ is defined  as
\begin{equation*}
\label{poisson-bracket}
\{K, H\} = \int \big( \partial_\eta H \partial_\xi K - \partial_\xi H \partial_\eta K \big) dx.
\end{equation*}
Assuming in addition that $K$ and $H$ are real-valued, the Poisson bracket takes the form 
\begin{equation}
\label{poisson-bracket-formula}
\begin{aligned}
\{K, H\} & = \int \big( \partial_{\eta_{{\rm k}_1}} H \partial_{\xi_{{\rm k}_2}} K - \partial_{\xi_{{\rm k}_1}} H \partial_{\eta_{{\rm k}_2}} K \big) \delta_{12} dk_{12}, \\
& = \frac{1}{i} \int \big( \partial_{z_{{\rm k}_1}} H \partial_{\bar z_{-{\rm k}_2}} K - \partial_{\bar z_{- {\rm k}_1}} H \partial_{z_{{\rm k}_2}} K \big) \delta_{12} dk_{12}.
\end{aligned}
\end{equation}

\subsection{Canonical transformation}

We first construct a  transformation that eliminates non-resonant terms from the cubic Hamiltonian \eqref{H3-fourier-in-z}. More precisely, we are  looking for a canonical transformation of the physical variables $(\eta,\xi)$
$$
\tau: w = \begin{pmatrix}
\eta \\ \xi
\end{pmatrix} \longmapsto w' =  \begin{pmatrix}
\eta' \\ \xi'
\end{pmatrix},
$$
defined in a neighborhood of the origin, such that the transformed Hamiltonian satisfies
$$
\H'(w') = \H(\tau^{-1}(w')), \quad 
\partial_t w' = J \, \nabla \H'(w'),
$$
and reduces to
$$
\H'(w') = \H^{(2)}(w') +  Z^{(3)} + Z^{(4)} + \ldots +  Z^{(m )} + R^{(m+1)},
$$
where $Z^{(j )}$ only consists of resonant terms  of order $j$ and $R^{(m+1)}$ is the remainder term at order $m$ \cite{CGS21b,CS16}. 
The transformation $\tau$ is obtained  as a Hamiltonian flow $\psi$ from ``time'' $s=-1 $ to ``time'' $s=0$ governed by
$$
\partial_s \psi = J \, \nabla K(\psi), \quad \psi(w')|_{s=0} = w' , \quad 
\psi(w')|_{s=-1} = w ,
$$
and associated to an auxiliary Hamiltonian $K$.
Such a transformation is canonical and preserves the Hamiltonian structure of the system. The Hamiltonian $\H'$ satisfies 
$\H'(w') = \H(\psi(w'))|_{s=-1}$ and its Taylor expansion around $s=0$ is
$$
\H'(w') = \H(\psi(w'))|_{s=0} - \frac{d \H}{ds}(\psi(w'))|_{s=0} + \frac{1}{2} \frac{d^2 \H}{ds^2}(\psi(w'))|_{s=0} - \ldots
$$
Abusing  notations, we  further  drop the primes and use $w = (\eta,\xi)^\top $ 
to denote the new variable $w'$.  Terms in this expansion can be expressed using Poisson brackets as 
\begin{equation*}
\begin{aligned}
\H(\psi(w))|_{s=0} & = \H(w), \\
\frac{d \H}{ds}(\psi(w))|_{s=0} & = \int \left( \partial_\eta \H \partial_s \eta + 
\partial_\xi \H \partial_s \xi \right) dx, \\    
& = \int \left( \partial_\eta \H \partial_\xi K - \partial_\xi \H \partial_\eta K \right) dx = \{K, \H\}(w),
\end{aligned}
\end{equation*}
and similar expressions  for higher-order $s$-derivatives. The Taylor expansion of $H'$ around $s=0$ now has the form 
$$
\H'(w) = \H(w) - \{K, \H\}(w) + \frac{1}{2} \{K, \{K, \H\}\}(w) - \ldots
$$
Substituting this transformation into  the expansion (\ref{HH}) of $\H$, we obtain
\begin{align*}
\H'(w) & = \H^{(2)}(w) + \H^{(3)}(w) + \ldots \\
& \quad - \{K,\H^{(2)} \} (w) -\{K, \H^{(3)}\} (w)  -\{K, \H^{(4)}\} (w) - \ldots \\
 & \quad + \frac{1}{2}\{K,\{K,\H^{(2)} \}\} (w) +\frac{1}{2} \{K,\{K,\H^{(3)}\}\} (w) + \ldots  
\end{align*} 
If $K$ is homogeneous of degree $m$ and $\H^{(n)}$ is homogeneous of degree $n$, then $\{K,\H^{(n)} \}$ is of degree $m+n-2$. 
Thus, if we construct an auxiliary Hamiltonian $K=K^{(3)}$ that is homogeneous of degree $3$ and satisfies the relation 
\begin{equation}\label{cohomological-relation}
   \H^{(3)}- \{K^{(3)},\H^{(2)} \}  = 0,
\end{equation}
we will have eliminated all cubic terms from the transformed Hamiltonian $\H'$.
We can repeat this process at each order.

\subsection{Third-order Birkhoff normal form}

We recall that the complex symplectic coordinates $z_{j}$ and $\bar{z}_{-j}$ diagonalize
the coadjoint operator
${\rm coad}_{\H^{(2)}} := \{\cdot, \H^{(2)}\}$, that is, the linear operation of taking Poisson brackets with
$H^{(2)}$ \cite{CS16,GKS22-siam}. When applying the operator to monomial terms such as 
$
\mathcal{I} := \int z_1 z_2 \bar z_{-3} \delta_{123} d{ k}_{123},
$
we have 
\begin{equation} \label{ad-H2}
\{ \mathcal{I}, \H^{(2)}\}  =  
i \int (\omega_1 + \omega_2 - \omega_3)   z_1 z_2 \bar z_{-3}  \delta_{123} d{k}_{123}.
\end{equation}

We use the diagonal property as in \eqref{ad-H2} to find the auxiliary Hamiltonian $K^{(3)}$ from \eqref{cohomological-relation}. The presence of resonant triads \eqref{resonance-set} does not allow us to eliminate $\H^{(3)}$ completely by virtue of \eqref{cohomological-relation}. Instead, we are only able to find $K^{(3)}$ such that 
\begin{equation}
\label{cohomol-eqn-nonresonant}
\{K^{(3)},\H^{(2)} \} = \H_{\rm NoRes}^{(3)},
\end{equation}
where $\H_{\rm NoRes}^{(3)}$ stands for the non-resonant part of the third-order Hamiltonian.  
Explicitly, we  define the non-resonant part of $\H^{(3)}$ as $\H_{\rm NoRes}^{(3)} := \H^{(3)} - \H_{\rm Res}^{(3)}$, 
where 
\begin{equation}
\label{H3-resonant}
\H_{\rm Res}^{(3)} = \frac{1}{8\sqrt{\pi}} \int \chi_{\mathcal{N}_\mu} (k_1, k_2, k_3) S_{123} (z_1 \overline{z}_{-2} z_3 + \overline{z}_{-1} {z}_{2} \overline{z}_{-3} )  \delta_{123} d{k}_{123},
\end{equation}
is a resonant part of \eqref{H3-fourier-in-z}, 
\begin{equation}
\label{S_123-defn}
S_{123} = S(k_1, k_2, k_3) := (k_1 k_3 + |k_1| |k_3|) \frac{a_1 a_3}{a_2},
\end{equation}
and $\chi_{\mathcal{N}_\mu}$ is the characteristic function defined in \eqref{characteristic-fnct-defn}.

\begin{proposition}
The cohomological equation \eqref{cohomol-eqn-nonresonant} has a unique solution $K^{(3)}$ which, in complex symplectic coordinates, is 
\begin{equation}
\begin{aligned}
\label{K3-fourier-z}
K^{(3)} = ~ \frac{1}{8i\sqrt{\pi}} \int &
S_{123} \Big[
\frac{z_1 z_2 z_3- \bar z_{-1} \bar z_{-2} \bar z_{-3} }{\omega_1 + \omega_2+ \omega_3} - 
2 \frac{z_1 z_2 \bar z_{-3}- \bar z_{-1} \bar z_{-2} z_3}{\omega_1+ \omega_2 -\omega_3} \\
&\qquad + \left(1-\chi_{\mathcal{N}_\mu} (k_1, k_2, k_3) \right)
\frac{ z_1\bar z_{-2}  z_3-  \bar z_{-1} z_2 \bar z_{-3} }{\omega_1 -\omega_2 +\omega_3}  \Big] \delta_{123} d{k}_{123}.
\end{aligned}
\end{equation} 
Alternatively, in the variables $(\eta,\xi)$, $K^{(3)}$ has the form
\begin{equation}
\label{K3-fourier-eta-xi}
\begin{aligned}
K^{(3)} = \frac{1}{\sqrt{2\pi}} \int & \left( P_{123}~ \eta_1 \eta_2 \xi_3 +  Q_{123}~ \eta_1 \xi_2 \eta_3 + R_{123}~  
\xi_1 \xi_2 \xi_3 \right) \delta_{123} d{k}_{123},
\end{aligned}
\end{equation}
where the  denominator $d_{123}$ is given in \eqref{d123-def}
and the coefficients are 
\begin{equation*}
\begin{aligned}
& P_{123} = \frac{1 + \sgn{k_1} \sgn{k_3}}{4 \widetilde d (k_1, k_3)} a_1^2 \left( 4 \omega_1 (\omega_1^2 - \omega_2^2 - \omega_3^2) - \chi_{\mathcal{N}_\mu} (k_1, k_2, k_3) \Pi_{123} \right), \\
& Q_{123} = \frac{1 + \sgn{k_1} \sgn{k_3}}{8 \widetilde d (k_1, k_3)} \left( \frac{a_1^2 a_3^2}{a_2^2} \right) \left( 8 \omega_1 \omega_2 \omega_3 + \chi_{\mathcal{N}_\mu} (k_1, k_2, k_3) \Pi_{123} \right), \\
& R_{123} = \frac{1 + \sgn{k_1} \sgn{k_3}}{8 \widetilde d (k_1, k_3)} \left( \frac{1}{a_2^2} \right) \left( 4 \omega_2 (\omega_1^2 - \omega_2^2 + \omega_3^2) - \chi_{\mathcal{N}_\mu} (k_1, k_2, k_3) \Pi_{123} \right),
\end{aligned}
\end{equation*}
with $\Pi_{123} := (\omega_1 + \omega_2 + \omega_3)(\omega_1 + \omega_2 - \omega_3) (\omega_1 - \omega_2 - \omega_3)$.
\end{proposition}

\begin{proof}
The derivations of \eqref{K3-fourier-z} and \eqref{K3-fourier-eta-xi} are based on straightforward computations. The expression \eqref{K3-fourier-z} is obtained by applying the diagonal property \eqref{ad-H2} to the non-resonant part of the cubic Hamiltonian, $\H_{\rm NoRes}^{(3)}$. The expression \eqref{K3-fourier-eta-xi} is derived from \eqref{K3-fourier-z} by substituting the relation \eqref{eta-xi-to-z-mapping}. 
\end{proof}

We point out that the coefficients $P_{123}, Q_{123}$ and $R_{123}$ are well defined in the neighborhood of the resonant triads \eqref{resonance-set}. This is due to the  construction of $K^{(3)}$ in \eqref{K3-fourier-z} which avoids the resonant triads that make the denominator $\omega_1-\omega_2 + \omega_3$ equal to $0$.  This is the principal  reason why  we  solve \eqref{cohomol-eqn-nonresonant} to find $K^{(3)}$ 
rather than the full relation $\{ K^{(3)}, \H^{(2)} \} = \H^{(3)}$ in \eqref{cohomological-relation}. The latter equation would lead to an ill-defined $K^{(3)}$ with vanishing denominator  $\omega_1-\omega_2 + \omega_3$, which we cannot handle.

The third-order normal form defining the new coordinates is obtained as the solution map at $s=0$ of 
the Hamiltonian flow 
\begin{equation*}
\label{NFT-flow-fourier-initial}
\partial_s   
\begin{pmatrix}
\eta \\ \xi
\end{pmatrix} = 
\begin{pmatrix}
0 & 1 \\
-1 & 0
\end{pmatrix}
\begin{pmatrix}
\partial_{\eta}K^{(3)} \\
\partial_{\xi} K^{(3)}
\end{pmatrix},
\end{equation*}
with initial condition at $s= -1$ being the original variables $(\eta,\xi)$.  Equivalently, in Fourier coordinates,
\begin{align}         \label{Eq:Hamiltonflow-Fourier}       
\partial_s \eta_{-k} = \partial_{\xi_k} K^{(3)},  \quad 
\partial_s \xi_{-k} 
= - \partial_{\eta_k} K^{(3)},
\end{align}
where
\begin{align*}
\partial_{\xi_k}   K^{(3)} & = \frac{1}{\sqrt{2\pi}} \int \Big(   \left( P_{12k} + Q_{1k2} \right)  \eta_1 \eta_2 + \left(R_{12k} + R_{2k1} + R_{k12} \right) \xi_1 \xi_2 \Big) \delta_{12k} d{k}_{12}, \\  
\partial_{\eta_k}   K^{(3)} & = \frac{1}{\sqrt{2\pi}} \int \left(  P_{1k2} + P_{k12} +Q_{12k} + Q_{k21}  \right)   \eta_1 \xi_2 \delta_{12k} d{k}_{12}, 
\end{align*}
by virtue of \eqref{K3-fourier-eta-xi}.

\section{Reduced Hamiltonian} \label{section:reduced-Hamil}

The new Hamiltonian $\H'$ obtained after applying the third-order normal form transformation has the form
\begin{equation}
\label{reduced-hamiltonian-0}
\begin{aligned}
\H(w) & = \H^{(2)}(w) + \H^{(3)}_{\rm Res}(w) + \H^{(4)}(w) - \{K^{(3)}, \H^{(3)}\}(w) + 
\frac{1}{2} \{K^{(3)}, \{K^{(3)}, \H^{(2)}\}\}(w) + R^{(5)}, \\
& = \H^{(2)}(w) + \H^{(3)}_{\rm Res}(w) + \H_+^{(4)}(w) + R^{(5)},
\end{aligned}
\end{equation}
where $R^{(5)}$ denotes all terms of order $5$ and higher, and $\H^{(4)}_+$ is the new fourth-order term
\begin{equation}
\label{new-H4-formula}
\H^{(4)}_+ = \H^{(4)} - \frac{1}{2} \{K^{(3)}, \H^{(3)}_{\rm NoRes}\} - \{K^{(3)}, \H^{(3)}_{\rm Res}\},
\end{equation}
where we have used the relation \eqref{cohomol-eqn-nonresonant}. 
Note that the presence of resonant terms did not eliminate all homogeneous cubic terms in $\H$, and the resonant terms also contribute to the quartic part of the Hamiltonian. 

In view of the forthcoming modulational Ansatz, the only important quartic terms are given by
\begin{equation}
\label{fourth-order-H-R}
\H_{+R}^{(4)} =  \int T z_1 z_2 \overline z_{3} \overline z_{4} \delta_{1+2-3-4} d{k}_{1234}.
\end{equation}
The other quartic terms 
will not contribute under this Ansatz due to homogenization; we refer to Section 4 of \cite{GKS22-siam} for details. In Section \ref{section:vanishing-resonant-cubic-term}, we show that  the resonant cubic Hamiltonian $\H^{(3)}_{\rm Res}$ in \eqref{reduced-hamiltonian-0} can also  be ruled out of the computations. 
As a result, the Hamiltonian $\H$ in \eqref{reduced-hamiltonian-0}    is at  leading order   
\begin{equation*}
\H(w) = \H^{(2)}(w) + \H_{+R}^{(4)}(w) +  
\widetilde{R}.
\end{equation*}

Denoting
\begin{equation} \label{decompH4}
\H_R^{(4)} = \int T_0 z_1 z_2 \overline z_{3} \overline z_{4} \delta_{1+2-3-4} d{ k}_{1234}, 
\end{equation}
\begin{equation} 
\label{decompH4-2}
\begin{aligned}
&\{K^{(3)}, \H^{(3)}_{\rm NoRes}\}_R = \int T_{\rm NoRes} z_1 z_2 \overline z_{3} \overline z_{4} \delta_{1+2-3-4} d{k}_{1234}, \\
& 
\{K^{(3)}, \H^{(3)}_{\rm Res}\}_R = \int T_{\rm Res} z_1 z_2 \overline z_{3} \overline z_{4} \delta_{1+2-3-4} d{k}_{1234},   
\end{aligned}
\end{equation}
the contributions of $zz\overline{z} \overline{z}$-type monomials to the terms of \eqref{new-H4-formula},
we have  that $T$ appearing  in \eqref{fourth-order-H-R} identifies to 
\begin{equation}
\label{T-formula}
T = T_0 - \frac{1}{2} T_{\rm NoRes} - T_{\rm Res}.
\end{equation}

\subsection{Explicit computations of $T_0$, $T_{\rm Res}$ and $T_{\rm NoRes}$}

Here, we provide 
precise formulas for the coefficients $T_0$, $T_{\rm Res}$ and $T_{\rm NoRes}$  
appearing in \eqref{decompH4}--\eqref{decompH4-2}.

\begin{proposition}
\label{lemma-T1-coeff}
We have $T_0 = T_0^{(1)} + T_0^{(2)}$, where
\begin{equation}
\label{T-0}
\begin{aligned}
& T_0^{(1)} = -V_{12(-3)(-4)} -V_{(-4)(-3)21} -V_{1(-3)2(-4)} -V_{(-4)2(-3)1} +V_{1(-4)(-3)2} +V_{(-3)21(-4)},\\[2pt]
& T_0^{(2)} = \frac{k_1 k_2 k_3 k_4}{32 \pi a_1 a_2 a_3 a_4} \big( 5\D (k_1k_2 -k_1k_3 - k_1k_4 - k_2k_3 - k_2k_4 + k_3k_4) - 3 \P \big),
\end{aligned}
\end{equation}
with 
\begin{equation}
\label{D-1234-defn}
\begin{aligned}
& V_{1234} = \frac{a_1 a_4}{32 \pi a_2 a_3} |k_1| |k_4| (|k_1|+|k_4| - 2|k_3+k_4|).
\end{aligned}    
\end{equation}

\begin{proof}
The proof is given in Appendix \ref{section-lemma-T1-coeff}. 
\end{proof}

\end{proposition}

\begin{proposition}
\label{lemma-T-nores-coeff}
Let $B_{123} := S_{123} \left(1-\chi_\mathcal{N_\mu} (k_1, k_2, k_3)\right)$ with $S_{123}$ and $\chi_{\mathcal{N}_\mu}$  given in \eqref{S_123-defn} and  \eqref{characteristic-fnct-defn} respectively. Then,  
$T_{\rm NoRes} =  T_{\rm NoRes}^{(1)} +  T_{\rm NoRes}^{(2)} +  T_{\rm NoRes}^{(3)}$ with
\begin{equation}
\label{T-nores-1}
\begin{aligned}
 T_{\rm NoRes}^{(1)} = & ~\frac{1}{64 \pi}  \big(S_{(-1-2)12} + S_{2(-1-2)1} + S_{12(-1-2)}\big)
 \left(S_{(-3-4)34} + S_{4(-3-4)3} + S_{34(-3-4)}\right) \\
 & \qquad \quad \times 
\left( \frac{1}{\omega_{1}+ \omega_{2}+ \omega_{1+2}} + \frac{1}{\omega_{3}+ \omega_{4}+ \omega_{3+4}} \right) \\
& + \frac{1}{16\pi} \big( S_{(3-1)(-3)1} + S_{1(3-1)(-3)} \big) \big( S_{(4-2)2(-4)} + S_{(-4)(4-2)2} \big) \\
 & \qquad \quad \times 
\left( \frac{1}{\omega_{3-1}+ \omega_{3}- \omega_{1}} + \frac{1}{\omega_{4-2}+ \omega_{2}- \omega_{4}} \right) \\
& - \frac{1}{16\pi} S_{12(-1-2)} S_{34(-3-4)} 
\left( \frac{1}{\omega_{1}+ \omega_{2}- \omega_{1+2}} + \frac{1}{\omega_{3}+ \omega_{4}- \omega_{3+4}} \right),
\end{aligned}
\end{equation}
\begin{equation}
\label{T-nores-2}
\begin{aligned}
T_{\rm NoRes}^{(2)} = &~ \frac{1}{32 \pi} \left( S_{12(-1-2)} B_{3(-3-4)4} + S_{34(-3-4)} B_{1(-1-2)2} \right) \left( \frac{1}{\omega_1 + \omega_2 - \omega_{1+2}} + \frac{1}{\omega_3 + \omega_4 - \omega_{3+4}} \right)\\
& - \frac{1}{16\pi} B_{(4-2)(-4)2} \left( S_{(3-1)(-3)1} + S_{(-3)(3-1)1} \right) \left( \frac{1}{\omega_{3-1} + \omega_3-\omega_1} + \frac{1}{\omega_{4-2} + \omega_2-\omega_4} \right)\\
& - \frac{1}{16\pi} B_{(3-1)1(-3)} \left( S_{(4-2)2(-4)} + S_{2(4-2) (-4)} \right) \left( \frac{1}{\omega_{3-1} + \omega_3-\omega_1} + \frac{1}{\omega_{4-2} + \omega_2-\omega_4} \right),
\end{aligned}
\end{equation}
and 
\begin{equation}
\label{T-nores-3}
\begin{aligned}
T_{\rm NoRes}^{(3)} = &~ - \frac{1}{64 \pi} B_{1(-1-2)2} B_{3(-3-4)4} \left( \frac{1}{\omega_{1} + \omega_2 - \omega_{1+2}} + \frac{1}{\omega_{3} + \omega_4 - \omega_{3+4}} \right)\\
& + \frac{1}{16 \pi} B_{(3-1)1(-3)} B_{(4-2)(-4)2} \left( \frac{1}{\omega_{3-1} + \omega_3 - \omega_1} + \frac{1}{\omega_{4-2} + \omega_2 - \omega_4} \right).
\end{aligned}
\end{equation}
Here, based on \eqref{S_123-defn}, $S_{(3-1)(-3)1}$ reads 
\begin{equation}
\label{S-3minus1-defn}
S_{(3-1)(-3)1} = \big( (k_3-k_1)k_1 + |k_3-k_1| |k_1| \big) \frac{a(k_3-k_1) a(k_1)}{a(-k_3)},
\end{equation}
and $B_{(3-1)(-3)1} = S_{(3-1)(-3)1} \big( 1- \chi_{\mathcal{N}_\mu} (k_3-k_1, -k_3, k_1) \big)$. 

\end{proposition}

\begin{proof}
The proof is given in Appendix \ref{section-lemma-T-nores-coeff}.
\end{proof}

\begin{proposition}
\label{lemma-T-res-coeff}
Recall that $S_{123}$ given by \eqref{S_123-defn} and $B_{123} := S_{123} \left(1-\chi_{\mathcal{N}_\mu} (k_1, k_2, k_3)\right)$. Then, $T_{\rm Res} = T_{\rm Res}^{(1)} + T_{\rm Res}^{(2)}$ with 
\begin{equation}
\label{T-res-1}
\begin{aligned}
T_{\rm Res}^{(1)} =&~  \frac{1}{32\pi} \left(  \frac{S_{12(-1-2)}}{\omega_1 + \omega_2 - \omega_{1+2}}   [\chi S]_{3(-3-4)4} + \frac{S_{34(-3-4)}}{\omega_3 + \omega_4 - \omega_{3+4}}  [\chi S]_{1(-1-2)2} \right)\\
& - \frac{1}{16 \pi (\omega_{3-1} + \omega_3 - \omega_{1})} \left( S_{(3-1)(-3)1} + S_{(-3)(3-1)1} \right) [\chi S]_{(4-2)(-4)2}\\
& - \frac{1}{16 \pi (\omega_{4-2} + \omega_2 - \omega_{4})} \left( S_{(4-2)2(-4)} + S_{2(4-2)(-4)} \right) [\chi S]_{(3-1)1(-3)},
\end{aligned}
\end{equation}
and 
\begin{equation}
\label{T-res-2}
\begin{aligned}
T_{\rm Res}^{(2)} =& \frac{1}{16\pi} \left( \frac{B_{(3-1)1(-3)}}{\omega_{3-1} + \omega_3 - \omega_{1}} [\chi S]_{(4-2)(-4)2} + \frac{B_{(4-2)(-4)2}}{\omega_{4-2} + \omega_2 - \omega_{4}} [\chi S]_{(3-1)1(-3)} \right)\\[3pt]
& - \frac{1}{64 \pi} \left( \frac{B_{1(-1-2)2}}{ \omega_1 + \omega_2 - \omega_{1+2}} [\chi S]_{3(-3-4)4} + \frac{B_{3(-3-4)4}}{ \omega_3 + \omega_4 - \omega_{3+4}} [\chi S]_{1(-1-2)2}  \right), 
\end{aligned}
\end{equation}
where we define 
\begin{equation*}
[\chi S]_{123} := B_{123} - S_{123} =  \chi_{\mathcal{N}_\mu} (k_1, k_2, k_3) S_{123}.  
\end{equation*}
\end{proposition}

\begin{proof}
The proof is similar to the proof of Proposition \ref{lemma-T-nores-coeff} in Appendix \ref{section-lemma-T-nores-coeff}. In particular, using the decomposition of $K^{(3)}$ given by terms \eqref{K3-0-chi-defn}, we write
\begin{equation*}
\{K^{(3)}, \H_{\rm Res}^{(3)}\} = \{K^{(3)}_0, \H_{\rm Res}^{(3)}\} + \{K^{(3)}_\chi, \H_{\rm Res}^{(3)}\}. 
\end{equation*}
Then the first Poisson bracket leads to the coefficient \eqref{T-res-1}, while the second bracket implies \eqref{T-res-2}. 
\end{proof}




\subsection{Modulational Ansatz}

We restrict ourselves to solutions in the form of near-monochromatic waves with carrier
wavenumber $k_0 > 0$.
For this purpose, we introduce the modulational Ansatz 
\begin{equation}
\label{modulation}
k = k_0 + \varepsilon \lambda, \quad \text{where}  \quad  \frac{\lambda}{k_0} = \calO(1),\quad \varepsilon \ll 1,
\end{equation}
and, accordingly, a function $U$ is defined in the Fourier space as 
\begin{equation}
\label{U-in-fourier}
U(\lambda) = z(k_0 + \varepsilon \lambda), \quad \overline U(\lambda) = \overline z(k_0 + \varepsilon \lambda),
\end{equation}
where  the time dependence is omitted. In the physical space,
\begin{equation}
\label{z-u-relation-physical}
\begin{aligned}
z(x) & = \frac{1}{\sqrt{2\pi}} \int z(k) e^{ikx} dk =  \frac{\varepsilon}{\sqrt{2\pi}} \int U(\lambda) e^{i k_0 x} e^{i \lambda \varepsilon x} d\lambda = \varepsilon \, u(X) e^{ik_0 x}, 
\end{aligned}
\end{equation}
where $u$, as a function of  the long  scale  $X = \varepsilon \, x$, is the inverse Fourier transform of $U$. After the change of variables \eqref{modulation}, the following integral identity for any function $C$ holds
\begin{equation}
\label{quartic-z-to-quartic-U-integral}
\begin{aligned}
& \int C(k_1, k_2, k_3, k_4) z_1 z_2 \overline z_3 \overline z_4 \delta_{1+2-3-4} dk_{1234} \\
& = \varepsilon^3 \int \widetilde C(\lambda_1, \lambda_2, \lambda_3, \lambda_4) U_1 U_2 \overline U_3 \overline U_4 \delta_{1+2-3-4}^{(\lambda)} d\lambda_{1234},
\end{aligned}
\end{equation}
where $U_j$ stands for $U(\lambda_j)$, the Dirac distribution on the right-hand side is defined as $\delta_{1+2-3-4}^{(\lambda)} = (1/2\pi) \int e^{-i (\lambda_1+ \lambda_2 - \lambda_3 - \lambda_4)X} dX$ and $C(k_1, k_2, k_3, k_4) = \widetilde C(\lambda_1, \lambda_2, \lambda_3, \lambda_4)$.


\subsection{Vanishing of the resonant cubic term $\H^{(3)}_{\rm Res}$}
\label{section:vanishing-resonant-cubic-term}

We show that, under the modulational regime \eqref{modulation}, the resonant cubic Hamiltonian $\H^{(3)}_{\rm Res}$ given in \eqref{H3-resonant}
 is sufficiently small and can be neglected. The estimates will be done for a cubic term of type
\begin{equation*}
\mathcal{I} := \int T_{k_1, k_2, k_3} z_1 \overline{z}_{2} z_3 \delta(k_1-k_2+k_3) dk_{123},
\end{equation*}
which can be similarly repeated for each term of \eqref{H3-resonant}.
After the change of variables \eqref{modulation}, we have
\begin{equation*}
\mathcal{I} = \frac{\varepsilon^3}{2\pi} \int e^{-ik_0 x} \int \widetilde{T}_{\lambda_1, \lambda_2, \lambda_3} U_1 \overline{U}_2 U_3 e^{-i(\lambda_1 - \lambda_2 + \lambda_3) (\varepsilon x)} d\lambda_{123} dx,
\end{equation*}
where $\widetilde{T}_{\lambda_1, \lambda_2, \lambda_3} :={T}_{k_0+\varepsilon \lambda_1, k_0+ \varepsilon \lambda_2, k_0+ \varepsilon \lambda_3}$. We identify the inner integral above with the function $f(\varepsilon x)$ as
\begin{equation*}
f(\varepsilon x) := \int \widetilde{T}_{\lambda_1, \lambda_2, \lambda_3} U_1 \overline{U}_2 U_3 e^{-i(\lambda_1 - \lambda_2 + \lambda_3) (\varepsilon x)} d\lambda_{123}.
\end{equation*}
As a result, we have
\begin{equation}
\label{cubic-term-generic}
\mathcal{I} = \frac{\varepsilon^3}{2\pi} \int e^{-ik_0 x} f(\varepsilon x) dx. 
\end{equation} 
To show that the integral \eqref{cubic-term-generic} is negligible, we use the scale-separation Lemma 4.4 of \cite{GKS22-siam}.

\subsection{Quartic interactions in the modulational regime} 

We approximate the coefficients in Propositions \ref{lemma-T1-coeff}
and \ref{lemma-T-nores-coeff} under the modulational regime \eqref{modulation}.

\begin{lemma}
\label{lemma-T1-approximation}
Under the modulational Ansatz (\ref{modulation}), we have 
\begin{equation*}
\int T_0 z_1 z_2 \overline z_3 \overline z_4 \delta_{1+2-3-4} dk_{1234} = \varepsilon^3 \int \Big(c_0^l 
 + \varepsilon c_0^r  (\lambda_2 + \lambda_3) \Big)
U_1 U_2 \overline U_3 \overline U_4 \delta_{1+2-3-4}^{(\lambda)} d\lambda_{1234}  + \calO(\varepsilon^5),
\end{equation*}
where  $T_0$ is given in Proposition \ref{lemma-T1-coeff} and the coefficients are 
\begin{equation}
\label{c-0-coeff}
\begin{aligned}
& c_0^l = \frac{k_0^3}{8\pi} + \frac{k_0^6}{16 \pi \omega_0^2} \left(\frac{3}{2} \P - 5\D k_0^2 \right), \\
&c_0^r = \frac{3 k_0^2}{16\pi} + \frac{k_0^6}{16 \pi \omega_0^2} \left[ \left({\frac{3}{2} \P} - 5\D k_0^2 \right) \left(\frac{2}{k_0} + \frac{g+\P k_0^2 -3\D k_0^4}{2\omega_0^2} \right) -5\D k_0 \right].
\end{aligned}
\end{equation} 
\end{lemma}

\begin{proof}
The proof is based on the direct expansion of the coefficients \eqref{T-0} under the modulational regime \eqref{modulation} using the identities in \eqref{expansion-identities}. For more detailed computations, we refer to the proof of Lemma 5.2 in \cite{GKS22-jfm}. 
\end{proof}

\begin{lemma}
\label{lemma-T2-approximation}
Under the modulational Ansatz (\ref{modulation}), we have  
\begin{equation}
\label{T-nores-1-approximation}
\begin{aligned}
\int T_{\rm NoRes}^{(1)} z_1 z_2 \overline z_3 \overline z_4 \delta_{1+2-3-4} dk_{1234} & = \varepsilon^3 \int 
\big(c_1^l  + \varepsilon c_1^r  (\lambda_2 + \lambda_3) \big) U_1 U_2 \overline U_3 \overline U_4 \delta_{1+2-3-4}^{(\lambda)} d\lambda_{1234} \\
& \quad + \frac{\varepsilon^4 k_0^2}{4\pi} \int   |\lambda_1 - \lambda_3| 
U_1 U_2 \overline U_3 \overline U_4 \delta_{1+2-3-4}^{(\lambda)} d\lambda_{1234} +  \calO(\varepsilon^5),
\end{aligned}
\end{equation}


\begin{equation}
\label{T-nores-2-approximation}
\begin{aligned}
\int T_{\rm NoRes}^{(2)} z_1 z_2 \overline z_3 \overline z_4 \delta_{1+2-3-4} dk_{1234} =  
\calO(\varepsilon^5),
\end{aligned}
\end{equation}
\begin{equation}
\label{T-nores-3-approximation}
\begin{aligned}
& \int T_{\rm NoRes}^{(3)} z_1 z_2 \overline z_3 \overline z_4 \delta_{1+2-3-4} dk_{1234} \\
& = \varepsilon^3 \int \left(1-\chi_{\mathcal{N}_\mu} (k_1, -k_1-k_2, k_2)\right) \left(1-\chi_{\mathcal{N}_\mu} (k_3, -k_3-k_4, k_4)\right) \\
& \qquad \quad \times \big(c_2^l + \varepsilon c_2^r (\lambda_2 + \lambda_3)\big) 
 U_1 U_2 \overline U_3 \overline U_4 \delta_{1+2-3-4}^{(\lambda)} d\lambda_{1234} \\
& \quad + \frac{\varepsilon^4 k_0^2}{4\pi} \int \left(1-\chi_{\mathcal{N}_\mu} (k_3-k_1, k_1, -k_3)\right) \left(1-\chi_{\mathcal{N}_\mu} (k_4-k_2, -k_4, k_2)\right) \\
& \qquad \qquad \quad \times (1+ \sgn{\lambda_1-\lambda_3}) |\lambda_1 - \lambda_3|
U_1 U_2 \overline U_3 \overline U_4 \delta_{1+2-3-4}^{(\lambda)} d\lambda_{1234}  + \calO(\varepsilon^5),
\end{aligned}
\end{equation}
where
{\small
\begin{equation}
\label{c-l-r-coeff}
\begin{aligned}
&c_1^l = \frac{k_0^3 \omega_0^2}{4\pi \omega_{2k_0} (2\omega_0 + \omega_{2k_0})},\\[3pt]
& c_1^r = \frac{c_1^l}{2} \left( 
\frac{3g-5\P k_0^2 + 7\D k_0^4}{\omega_{0}^2} +
\frac{g+4\P k_0^2 -48 \D k_0^4}{\omega_{2k_0}^2} -
\frac{g-12\P k_0^2 +80 \D k_0^4}{\omega_{2k_0}(2\omega_0 + \omega_{2k_0})} - 
\frac{g-3\P k_0^2 +5 \D k_0^4}{\omega_{0}(2\omega_0 + \omega_{2k_0})}
\right),\\[3pt]
& c_2^l = -\frac{k_0^3 \omega_0^2}{4\pi \omega_{2k_0} (2\omega_0 - \omega_{2k_0})},\\[3pt]
& c_2^r = \frac{c_2^l}{2} \left( 
\frac{3g-5\P k_0^2 + 7\D k_0^4}{\omega_{0}^2} +
\frac{g+4\P k_0^2 -48 \D k_0^4}{\omega_{2k_0}^2} { +}
\frac{g-12\P k_0^2 +80 \D k_0^4}{\omega_{2k_0}(2\omega_0 - \omega_{2k_0})} - 
\frac{g-3\P k_0^2 +5 \D k_0^4}{\omega_{0}(2\omega_0 - \omega_{2k_0})}
\right).
\end{aligned}
\end{equation}}
\end{lemma}

\begin{proof}

Below we provide an outline of the computational steps. 

\noindent
\textbf{\textit{Proof of \eqref{T-nores-1-approximation}}}:
First, we note that the terms $S_{(-1-2)12}$, $S_{12(-1-2)}$, $S_{(-3-4)34}$ and $S_{34(-3-4)}$ in the expressions \eqref{T-nores-1}--\eqref{T-nores-3} vanish.  This can be shown by expanding the terms under the modulational Ansatz in view of identities in Appendix \ref{appendix-identities}. For example, for $k_1 = k_0+\varepsilon \lambda_1$ and $k_2 = k_0+\varepsilon \lambda_2$, one can see that 
\begin{equation}
\label{sign-product-expansion}
1 + \sgn {-k_1-k_2} \sgn {k_2} = 0,
\end{equation} 
which is sufficient to verify that $S_{(-1-2)12} = 0$.  
As a result, we rewrite $T_{\rm NoRes}^{(1)}$ as 
\begin{equation}
\label{T-nores-1-leading}
\begin{aligned}
T_{\rm NoRes}^{(1)} = & ~\frac{1}{64 \pi}  S_{2(-1-2)1}
S_{4(-3-4)3} 
\Big( \frac{1}{\omega_{1}+ \omega_{2}+ \omega_{1+2}} + \frac{1}{\omega_{3}+ \omega_{4}+ \omega_{3+4}} \Big) \\
& + \frac{1}{16\pi} S_{(3-1)(-3)1} 
S_{(4-2)2(-4)}  
\Big( \frac{1}{\omega_{3-1}+ \omega_{3}- \omega_{1}} + \frac{1}{\omega_{4-2}+ \omega_{2}- \omega_{4}} \Big).
\end{aligned}
\end{equation}
Using the expansions \eqref{expansion-identities}--\eqref{expansion-identities-1plus2} and the definitions \eqref{S_123-defn} and \eqref{S-3minus1-defn}, we get for $S_{2(-1-2)1}, S_{(3-1)(-3)1}$ the following: 
\begin{equation}
\label{expansion-S-2-1}
\begin{aligned}
S_{2(-1-2)1} & = S (k_2, -k_1-k_2, k_2), \\
& = \omega_0 k_0 \sqrt{\frac{2k_0}{\omega_{2k_0}}} \left(1 + \frac{\varepsilon (\lambda_1 + \lambda_2)}{4} \left( \frac{4}{k_0} + \frac{g+4\P k_0^2 - 48 \D k_0^4}{\omega_{2k_0}^2} - \frac{g+\P k_0^2 - 3 \D k_0^4}{\omega_{0}^2} \right) \right),   
\end{aligned}    
\end{equation}
and 
\begin{equation}
\begin{aligned}
S_{(3-1)(-3)1} & = k_0 g^{1/4} \varepsilon^{3/4} (1+\sgn{\lambda_3-\lambda_1}) |\lambda_3-\lambda_1|^{3/4}.
\end{aligned}    
\end{equation}
Performing similar computations on the remaining terms in \eqref{T-nores-1-leading}, its first line becomes 
\begin{equation*}
\begin{aligned}
\frac{1}{64 \pi}  S_{2(-1-2)1}
S_{4(-3-4)3} 
\Big( \frac{1}{\omega_{1}+ \omega_{2}+ \omega_{1+2}} + \frac{1}{\omega_{3}+ \omega_{4}+ \omega_{3+4}} \Big) = c_1^l + \frac{1}{2} c_1^r \varepsilon (\lambda_1+\lambda_2+\lambda_3+\lambda_4),
\end{aligned}
\end{equation*}
while the second line of \eqref{T-nores-1-leading} is equivalent to 
\begin{equation}
\label{second-line-nonlocal-expansion}
\begin{aligned}
& \frac{1}{16\pi} S_{(3-1)(-3)1} S_{(4-2)2(-4)} 
\Big( \frac{1}{\omega_{3-1}+ \omega_{3}- \omega_{1}} + \frac{1}{\omega_{4-2}+ \omega_{2}- \omega_{4}} \Big) \\
& \quad = \frac{\varepsilon k_0^2}{4\pi} (1+ \sgn{\lambda_3-\lambda_1}) |\lambda_3-\lambda_1|. 
\end{aligned}
\end{equation}
In view of \eqref{quartic-z-to-quartic-U-integral} and Lemma \ref{lemma-integral-identities-1}, this leads to \eqref{T-nores-1-approximation}.

\noindent
\textbf{\textit{Proof of \eqref{T-nores-2-approximation}}}:
Similarly, noting additionally that $S_{(-3)(3-1)1}$ and $S_{2(4-2)(-4)}$ vanish, the term $T_{\rm NoRes}^{(2)}$ reduces to
\begin{equation}
\label{T-nores-2-leading}
\begin{aligned}
T_{\rm NoRes}^{(2)} = & - \frac{1}{16\pi} B_{(4-2)(-4)2} S_{(3-1)(-3)1} \left( \frac{1}{\omega_{3-1} + \omega_3-\omega_1} + \frac{1}{\omega_{4-2} + \omega_2-\omega_4} \right)\\
& - \frac{1}{16\pi} B_{(3-1)1(-3)} S_{(4-2)2(-4)} \left( \frac{1}{\omega_{3-1} + \omega_3-\omega_1} + \frac{1}{\omega_{4-2} + \omega_2-\omega_4} \right).
\end{aligned}
\end{equation}
Due to the presence of $\delta_{1+2-3-4}$ inside
the right-hand side integral in \eqref{T-nores-2-approximation}, we will see that the leading terms of $T_{\rm NoRes}^{(2)}$ in \eqref{T-nores-2-leading} are zero.
Indeed, from the definitions of $S_{123}$ and $B_{123}$, the first line in \eqref{T-nores-2-leading} contains the factor $B_{(4-2)(-4)2} S_{(3-1)(-3)1}$ involving
\begin{equation}
\label{product-signs-T2-nores}
(1+\sgn{k_4-k_2} \sgn{k_2}) (1+\sgn{k_3-k_1} \sgn{k_1}).
\end{equation}
From the constraint $k_1+k_2-k_3-k_4=0$,  we have 
\begin{equation*}
k_1-k_3 = k_4-k_2 \implies \sgn{k_3-k_1} = - \sgn{k_4-k_2}. 
\end{equation*}
Since, under the modulational regime \eqref{modulation}, $\sgn{k_1} = \sgn{k_2} = +1$, direct computations show that
the expression \eqref{product-signs-T2-nores} vanishes, and so does the first line in \eqref{T-nores-2-leading}.

We follow similar steps for the leading term on the second line of \eqref{T-nores-2-leading}. The factor $B_{(3-1)1(-3)} S_{(4-2)2(-4)}$ involves
\begin{equation}
(1+\sgn{k_3-k_1} \sgn{-k_3}) (1+\sgn{k_4-k_2} \sgn{-k_4}),
\end{equation}
which vanishes under the constraint $k_1+k_2-k_3-k_4 = 0$. This leads to \eqref{T-nores-2-approximation}.

\noindent
\textbf{\textit{Proof of \eqref{T-nores-3-approximation}}}:
The computations for $T_{\rm NoRes}^{(3)}$ show the effect of resonances in our problem on the coefficients of the quartic Hamiltonian.

Recall the definition $B_{123} = (1-\chi_{\mathcal{N}_\mu} (k_1, k_2, k_3)) S_{123}$. Then, using expansions as in \eqref{expansion-S-2-1}, the first line of \eqref{T-nores-3} becomes 
\begin{equation}
\label{first-line-T-nores-3-expanded}
\begin{aligned}
& - \frac{1}{64 \pi} B_{1(-1-2)2} B_{3(-3-4)4}  \left( \frac{1}{\omega_{1} + \omega_2 - \omega_{1+2}} + \frac{1}{\omega_{3} + \omega_4 - \omega_{3+4}} \right) \\
& = (1-\chi_{\mathcal{N}_\mu} (k_1, -k_1-k_2, k_2)) (1-\chi_{\mathcal{N}_\mu} (k_3, -k_3-k_4, k_4))
\left( c_2^l + \frac{1}{2} c_2^r \varepsilon (\lambda_1+\lambda_2+\lambda_3+\lambda_4) \right). 
\end{aligned}
\end{equation}

For the second line in \eqref{T-nores-3}, we perform computations similar to \eqref{second-line-nonlocal-expansion}, and obtain that it is equal to
\begin{equation}
\label{second-line-nonlocal-expansion-cut-off}
\begin{aligned}
& \frac{1}{16 \pi} B_{(3-1)1(-3)} B_{(4-2)(-4)2} \left( \frac{1}{\omega_{3-1} + \omega_3 - \omega_1} + \frac{1}{\omega_{4-2} + \omega_2 - \omega_4} \right) \\
& = (1-\chi_{\mathcal{N}_\mu} (k_3-k_1, k_1, -k_3)) (1-\chi_{\mathcal{N}_\mu} (k_4-k_2, -k_4, k_2))
\frac{\varepsilon k_0^2}{4\pi} (1+ \sgn{\lambda_1-\lambda_3}) |\lambda_1-\lambda_3|. 
\end{aligned}
\end{equation}


\end{proof}

\begin{lemma}
\label{lemma-T3-approximation}
Under the modulational Ansatz (\ref{modulation}), we have
\begin{equation}
\label{T-res-approximation}
\begin{aligned}
& \int T_{\rm Res} z_1 z_2 \overline z_3 \overline z_4 \delta_{1+2-3-4} dk_{1234} \\
& =  
\frac{\varepsilon^4 k_0^2}{8\pi} \int 
\Big[ \chi_{\mathcal{N}_\mu} (k_3-k_1, k_1, -k_3) + \chi_{\mathcal{N}_\mu} (k_4-k_2, -k_4, k_2) \\
& \qquad \qquad \quad - 2 \chi_{\mathcal{N}_\mu} (k_3-k_1, k_1, -k_3) \chi_{\mathcal{N}_\mu} (k_4-k_2, -k_4, k_2)\Big] \\
& \qquad \qquad \qquad \times (1+ \sgn{\lambda_1-\lambda_3}) |\lambda_1 - \lambda_3| U_1 U_2 \overline U_3 \overline U_4 \delta_{1+2-3-4}^{(\lambda)} d\lambda_{1234}\\
& \quad + \frac{\varepsilon^3}{2} \int  \Big[ \chi_{\mathcal{N}_\mu} (k_1, -k_1-k_2, k_2) +  \chi_{\mathcal{N}_\mu} (k_3, -k_3-k_4, k_4) \\
& \qquad \qquad \quad - 2 \chi_{\mathcal{N}_\mu} (k_1, -k_1-k_2, k_2)  \chi_{\mathcal{N}_\mu} (k_3, -k_3-k_4, k_4) \Big] \\
& \qquad  \qquad \qquad \times (c_2^l + \varepsilon c_2^r (\lambda_2 + \lambda_3)) U_1 U_2 \overline U_3 \overline U_4 \delta_{1+2-3-4}^{(\lambda)} d\lambda_{1234}.
\end{aligned}
\end{equation}

\end{lemma}
\begin{proof}
The computational steps are similar to those in the proof of Lemma \ref{lemma-T2-approximation}, and we give an outline of the steps below. 

For $T_{\rm Res}^{(1)}$ in \eqref{T-res-1}, the first line can be neglected due to the presence of $S_{12(-1-2)}$ and $S_{34(-3-4)}$, both of which vanish due to the presence of a factor of type \eqref{sign-product-expansion}. For the second line, we also have $S_{(-3)(3-1)1} = 0$ and it reduces to 
\begin{equation*}
- \frac{1}{16 \pi (\omega_{3-1} + \omega_3 - \omega_{1})} S_{(3-1)(-3)1} [\chi S]_{(4-2)(-4)2},
\end{equation*}
which has a factor given by \eqref{product-signs-T2-nores} vanishing under the assumption $k_1+k_2-k_3-k_4 = 0$. The third line of \eqref{T-res-1} is treated similarly. 

For $T_{\rm Res}^{(2)}$, we follow the steps used in the proof of Lemma \ref{lemma-T2-approximation}. As a result, we get
\begin{equation*}
\begin{aligned}
& \frac{B_{(3-1)1(-3)}}{\omega_{3-1} +\omega_3 - \omega_1} [\chi S]_{(4-2)(-4)2} \\
& = 2 \varepsilon k_0^2 \left( 1 - \chi_{\mathcal{N}_\mu} (k_3-k_1, k_1, -k_3) \right) \chi_{\mathcal{N}_\mu} (k_4-k_2, -k_4, k_2) \left( 1 + \sgn{\lambda_1-\lambda_3} \right) |\lambda_1-\lambda_3|.
\end{aligned}
\end{equation*}
Combining this expression with a similar one for the second term of the first line in \eqref{T-res-2}, we obtain the first integral term of \eqref{T-res-approximation} under \eqref{quartic-z-to-quartic-U-integral}.  
Repeating the computation as in \eqref{expansion-S-2-1} for the terms in the second line of \eqref{T-res-2}, we get 
\begin{equation*}
\begin{aligned}
& \frac{-1}{64\pi}\left( \frac{B_{1(-1-2)2}}{ \omega_1 + \omega_2 - \omega_{1+2}} [\chi S]_{3(-3-4)4} + \frac{B_{3(-3-4)4}}{ \omega_3 + \omega_4 - \omega_{3+4}} [\chi S]_{1(-1-2)2} \right) \\
& \qquad  = \frac{1}{2} \Big[ \chi_{\mathcal{N}_\mu} (k_1, -k_1-k_2, k_2) +  \chi_{\mathcal{N}_\mu} (k_3, -k_3-k_4, k_4) \\
& \qquad \qquad - 2 \chi_{\mathcal{N}_\mu} (k_1, -k_1-k_2, k_2)  \chi_{\mathcal{N}_\mu} (k_3, -k_3-k_4, k_4) \Big] \\
& \qquad \qquad \quad \times (c_2^l + \varepsilon c_2^r (\lambda_2 + \lambda_3)),
\end{aligned}
\end{equation*}
which leads to the second line of \eqref{T-res-approximation}. 
\end{proof}

Based on Lemmas \ref{lemma-T1-approximation}--\ref{lemma-T3-approximation}, 
the reduced Hamiltonian $\H^{(4)}_+$ has the form
\begin{equation}
\label{H4-reduced}
\begin{aligned}
\H^{(4)}_+  &=  ~ \varepsilon^3
\int \frac{\alpha}{4\pi}
U_1 U_2 \overline{U}_3 \overline{U}_4 \delta_{1+2-3-4}^{(\lambda)} d\lambda_{1234}\\ 
& \quad + \varepsilon^4 \int \Big( 
\frac{\beta}{8\pi}  (\lambda_2 + \lambda_3)  -  \frac{k_0^2}{8\pi} \gamma  |\lambda_1 - \lambda_3| \Big)
U_1 U_2 \overline{U}_3 \overline{U}_4 \delta_{1+2-3-4}^{(\lambda)} d\lambda_{1234} 
 + \mathcal{O}(\varepsilon^5),
\end{aligned} 
\end{equation}
where, using $k_j = k_0 + \varepsilon \lambda_j$ from \eqref{modulation}, we denote
\begin{equation}
\label{alpha-value}
\begin{aligned}
\frac{\alpha}{4\pi}: & = c_0^l - \frac{1}{2}c_1^l - \frac{1}{2} \left[1 - \chi_{\mathcal{N}_\mu} (k_1, -k_1-k_2, k_2) \chi_{\mathcal{N}_\mu} (k_3, -k_3-k_4, k_4) \right] c_2^l,
\end{aligned}
\end{equation}

\begin{equation}
\label{beta-value}
\frac{\beta}{8\pi}:= c_0^r - \frac{1}{2}c_1^r - \frac{1}{2} \left[1 - \chi_{\mathcal{N}_\mu} (k_1, -k_1-k_2, k_2) \chi_{\mathcal{N}_\mu} (k_3, -k_3-k_4, k_4) \right] c_2^r,
\end{equation}
and
\begin{equation}
\label{gamma-value}
\begin{aligned}
\gamma = 1 + \left[1-\chi_{\mathcal{N}_\mu} (k_3-k_1, k_1, -k_3) \chi_{\mathcal{N}_\mu} (k_4-k_2, -k_4, k_2)\right] (1+ \sgn{\lambda_1-\lambda_3}).
\end{aligned}
\end{equation}

\section{Hamiltonian Dysthe equation}
\label{section:dysthe}

From the expressions \eqref{alpha-value}--\eqref{gamma-value}, it is clear that the coefficients in the Hamiltonian \eqref{H4-reduced} depend on $k_0$ via \eqref{c-0-coeff} and \eqref{c-l-r-coeff}, and
on the characteristic function at corresponding points inside the integral. In turn, the choice of $\mu$ in the definition of $\chi_{\mathcal{N}_\mu}$ in \eqref{characteristic-fnct-defn} and of $k_0$ in our modulational Ansatz \eqref{modulation} affect the evaluation of these characteristic functions. 
We concentrate on two examples of values of $k_0$, motivated by the three-wave resonant curve (figure \ref{figure-resonant-triads}) and the numerical simulations validating our asymptotics in Section \ref{section-numerical-results}. 
We choose $k_0=0.9$ and $k_0=2$, which lead to different values of characteristic functions appearing in \eqref{alpha-value}--\eqref{gamma-value}. In both cases, we set $\mu = \mathcal{O}(\varepsilon)$. Details are given in the next subsection. 


\subsection{Explicit formulas for $\alpha, \beta$ and $\gamma$}

\subsubsection{Case $k_0 = 0.9$}

Based on the construction of the neighborhood around the resonant set in \eqref{R-neighborhood}, we have the following identities: 
\begin{equation}
\label{assumption-center-k0-small}
\chi_{\mathcal{N}_\mu} (k_1, -k_1-k_2, k_2) = 0,
\end{equation}
and 
\begin{equation}
\label{assumption-around-zero-k0-small}
\chi_{\mathcal{N}_\mu} (k_4-k_2, -k_4, k_2) = 0 \quad \text{whenever} \quad \sgn{k_4-k_2} = +1. 
\end{equation}
The validity of \eqref{assumption-center-k0-small} follows from \eqref{modulation} and the choice $\mu = \mathcal{O}(\varepsilon)$:
\begin{equation}
\label{assumption-center-details-k0-small}
\begin{aligned}
& (k_1, k_2) \in \mathcal{B}_c (k_0) := \left\{ (x,y): x= k_0+\mathcal{O} (\varepsilon), y= k_0+\mathcal{O} (\varepsilon) \right\} \not \subseteq  \mathcal{C}_\mu \\
& \implies (k_1, -k_1-k_2, k_2) \notin \mathcal{N}_\mu.
\end{aligned}
\end{equation}
As for \eqref{assumption-around-zero-k0-small}, under 
\eqref{modulation}, we have $k_4-k_2 = \varepsilon (\lambda_4 - \lambda_2) = \mathcal{O}(\varepsilon)$ and $k_2 = k_0+\varepsilon \lambda_2$. Then, if $k_4-k_2 >0$, the points $(k_4-k_2, k_2)$ are located far away from the set $\mathcal{C}_\mu^+$ and so 
\begin{equation}
\label{assumption-around-zero-details-k0-small}
\begin{aligned}
& (k_4-k_2, k_2) \in \mathcal{B}_s (k_0) := \left\{ (x,y): 0\leq x \leq \mathcal{O}(\varepsilon), y = k_0 + \mathcal{O}(\varepsilon) \right\} \not \subseteq \mathcal{C}^+_\mu \\
& \implies (k_4-k_2, -k_4, k_2) \notin \mathcal{N}_\mu.
\end{aligned}
\end{equation}
As a result, the coefficients in \eqref{alpha-value}--\eqref{gamma-value} reduce to
\begin{equation}
\label{coefficients-k0-small}
\begin{aligned}
\frac{\alpha}{4\pi} = c_0^l - \frac{1}{2} (c_1^l + c_2^l), \quad 
\frac{\beta}{8\pi} =  c_0^r - \frac{1}{2} (c_1^r + c_2^r), \quad
\gamma = 2+\sgn{\lambda_1-\lambda_3}.
\end{aligned}
\end{equation}
When substituting $\gamma$ into the integral \eqref{H4-reduced}, the term $\sgn{\lambda_1-\lambda_3}$ will disappear since 
\begin{equation*}
\begin{aligned}
& \int \sgn{\lambda_1-\lambda_3} |\lambda_1 - \lambda_3| U_1 U_2 \overline{U}_3 \overline{U}_4 \delta_{1+2-3-4}^{(\lambda)} d\lambda_{1234} = \int (\lambda_1 - \lambda_3) U_1 U_2 \overline{U}_3 \overline{U}_4 \delta_{1+2-3-4}^{(\lambda)} d\lambda_{1234} \\
& = \frac{1}{2} \int (\lambda_1+\lambda_2 - \lambda_3 - \lambda_4) U_1 U_2 \overline{U}_3 \overline{U}_4 \delta_{1+2-3-4}^{(\lambda)} d\lambda_{1234} = 0,
\end{aligned}
\end{equation*}
where we have used the index rearrangement $(\lambda_1, \lambda_2, \lambda_3, \lambda_4) \to (\lambda_2, \lambda_1, \lambda_3, \lambda_4)$ and the delta condition $\lambda_1+\lambda_2-\lambda_3-\lambda_4 = 0$. 
As a consequence, we will simply write $\gamma = 2$ in this case.

\subsubsection{Case $k_0 = 2$}

If $k_0$ is sufficiently large, we have
\begin{equation}
\label{assumption-center-k0-big}
\chi_{\mathcal{N}_\mu} (k_1, -k_1-k_2, k_2) = 0,
\end{equation}
and 
\begin{equation}
\label{assumption-around-zero-k0-big}
\chi_{\mathcal{N}_\mu} (k_4-k_2, -k_4, k_2) = 1 \quad \text{whenever} \quad \sgn{k_4-k_2} = +1. 
\end{equation}
Equation \eqref{assumption-center-k0-big} holds due to estimates similar to \eqref{assumption-center-details-k0-small}. 
Equation \eqref{assumption-around-zero-k0-big} is different from \eqref{assumption-around-zero-k0-small} since, for a sufficiently large value of $k_0$, the points $(k_4-k_2, k_2)$ are located inside the set $\mathcal{C}_\mu^+$: 
\begin{equation}
\label{assumption-around-zero-details-k0-big}
(k_4-k_2, k_2) \in \left\{ (x,y): 0\leq x \leq \mathcal{O}(\varepsilon), y = k_0 + \mathcal{O}(\varepsilon) \right\} \subset \mathcal{C}^+_\mu \implies (k_4-k_2, -k_4, k_2) \in \mathcal{N}_\mu,
\end{equation}
and we get
\begin{equation}
\label{coefficients-k0-large}
\begin{aligned}
\frac{\alpha}{4\pi} = c_0^l - \frac{1}{2} (c_1^l + c_2^l), \quad 
\frac{\beta}{8\pi} =  c_0^r - \frac{1}{2} (c_1^r + c_2^r), \quad
\gamma = 1.
\end{aligned}
\end{equation}

Figure \ref{fig:assumptions} shows the location of points $(k_1, k_2)$ and $(k_4-k_2, k_2)$ with $k_4-k_2 >0$ for $k_0 = 0.9$ (figure \ref{fig:assumptions}a) and $k_0=2$ (figure \ref{fig:assumptions}b) relative to the neighborhood $\mathcal{C}_\mu$. In figure \ref{fig:assumptions}(a) for $k_0 = 0.9$, Box 1 represents the set $\mathcal{B}_s (k_0)$ of all possible values of $(k_4-k_2, k_2)$ with $k_4-k_2 >0$ under the modulational Ansatz \eqref{modulation}, according to \eqref{assumption-around-zero-details-k0-small}. Box 2 represents the set $\mathcal{B}_c (k_0)$ in \eqref{assumption-center-details-k0-small}. Similar sets are depicted in figure \ref{fig:assumptions}(b) for $k_0 = 2$.

\subsection{Derivation of the Dysthe equation}

The third-order normal form transformation has  eliminated all cubic terms from the Hamiltonian $\H$. In the modulational 
regime (\ref{modulation}), the reduced Hamiltonian (now denoted by $\H$) is  
\begin{equation*}
\label{reduced-H-fourier}
\H  = \H^{(2)} + \H_+^{(4)}, 
\end{equation*}
up to fourth order.
In the physical variables $(u, \overline u)$, it reads 
\begin{equation}
\label{reduced-H-for-dysthe}
\begin{aligned}
\H = &~ \varepsilon \int \overline u \, \omega (k_0 + \varepsilon D_X) \, u\, dX + \varepsilon^3 \frac{\alpha}{2} \int |u|^4 dX\\[3pt]
& + \varepsilon^4 \frac{\beta}{2} \int |u|^2 {\rm Im} (\overline{u} \partial_X u) dX 
- \varepsilon^4 \frac{\gamma k_0^2}{4} \int  |u|^2 |D_X| |u|^2 \, dX + \calO (\varepsilon^5),
\end{aligned}
\end{equation}
where $D_X = -i \, \partial_X$ for the slow spatial variable $X$.
Expanding the linear dispersion relation in \eqref{reduced-H-for-dysthe} around $k_0$ as 
\begin{equation}
\label{omega-taylor-expansion}
\omega (k_0+ \varepsilon D_X) = \omega_0 + \varepsilon \omega_0' D_X +  \frac{1}{2}\varepsilon^2 \omega_0'' D^2_X+ \frac{1}{6}\varepsilon^3 \omega_0''' D^3_X + \calO(\varepsilon^4),
\end{equation}
gives an alternate form  of the Hamiltonian $\H$ in physical variables
\begin{equation}
\label{reduced-H-for-dysthe-new}
\begin{aligned}
\H = & \int \Big[ \varepsilon\, \omega_0 \, |u|^2 + \varepsilon^2 \omega_0' {\rm Im}(\overline u \partial_X u)
+ \frac{1}{2} \varepsilon^3 \omega_0'' |\partial_X u|^2 + \frac{1}{2} \varepsilon^3 \alpha |u|^4 \\
& \qquad - \frac{1}{6} \varepsilon^4 \omega_0''' {\rm Im} (\overline u \partial_X^3 u) + \frac{1}{2} \varepsilon^4 \beta |u|^2 {\rm Im} (\overline{u} \partial_X u) 
- \frac{1}{4} \varepsilon^4 \gamma k_0^2 |u|^2 |D_X| |u|^2 \Big] dX +  \calO(\varepsilon^5).
\end{aligned}
\end{equation}
Coefficients in \eqref{omega-taylor-expansion} have the following expressions
\begin{align*}
\omega_0' := \partial_k \omega(k_0) & = \frac{g + 5 \D k_0^4 - 3 \P k_0^2}{2 \omega_0}, \\
\omega_0'' := \partial_k^2 \omega(k_0) & = \frac{15 \D^2 k_0^8 - 22 \D \P k_0^6 + 30 g \D k_0^4 + 3 \P^2 k_0^4 - 6 g \P k_0^2 - g^2}{4 \omega_0^{3/2}}, \\
\omega_0''' := \partial_k^3 \omega(k_0) & = 
3 \big( 5 \D^3 k_0^{12} - 13 \D^2 \P k_0^{10} - 5 g \D^2 k_0^8 + 15 \D \P^2 k_0^8 - 34 g \D \P k_0^6 + \P^3 k_0^6 \\
& \quad + 55 g^2 \D k_0^4 - 5 g \P^2 k_0^4 - 5 g^2 \P k_0^2 + g^3 \big)/\big( 8 \omega_0^{5/2} \big).
\end{align*}
The Hamiltonian system
\begin{equation}
\label{ww-hamiltonian-in-u}
\partial_t \begin{pmatrix}
u \\ \overline u
\end{pmatrix} = \varepsilon^{-1} \begin{pmatrix}
0 & -i \\ i & 0
\end{pmatrix} \begin{pmatrix}
\partial_u \H \\ \partial_{\overline u} \H
\end{pmatrix},
\end{equation}
implies
\begin{equation}  \label{Dysthe}
\begin{aligned}
i \, \partial_t u & = \varepsilon^{-1} \partial_{\overline u} \H, \\
& = \omega_0 u -i \varepsilon \omega_0' \partial_X u - \frac{1}{2} \varepsilon^2 \omega_0'' \partial_X^2 u + \varepsilon^2 \alpha |u|^2 u \\
& \quad + \frac{i}{6} \varepsilon^3 \omega_0''' \partial_X^3 u - i \varepsilon^3 \beta |u|^2 \partial_X u 
- \varepsilon^3 \frac{\gamma k_0^2}{2} u |D_X| |u|^2,
\end{aligned}
\end{equation}
which is a Hamiltonian Dysthe equation for two-dimensional hydroelastic waves on deep water. It describes modulated waves moving in the positive $x$-direction at group speed $\omega_0'$ as shown by the advection term. The nonlocal term $u|D_X| |u|^2$, a signature of the Dysthe equation, reflects the presence of the wave-induced mean flow as in the classical derivation using the method of multiple scales.
The coefficient of this mean-flow term is similar to that in the pure gravity case ($\D = 0$, $\P = 0$) except for the dependence on $\gamma$.

The first two terms on the right-hand side of \eqref{Dysthe} can be eliminated via phase invariance and reduction to a moving reference frame.
The latter is equivalent, in the framework of canonical transformations, to subtraction from $\H$ of a multiple of the momentum \eqref{momentum}
which reduces to
$$
I  = \int \eta \, (\partial_x \xi) dx = \int \Big[ k_0 |u|^2 + \varepsilon \, {\rm Im}(\overline u \partial_{X} u) \Big] dX,
$$
while the former is equivalent to subtraction from $\H$ of a multiple of the wave action 
\begin{equation} \label{action}
M = \varepsilon  \int |u|^2 \, dX,
\end{equation}
which is conserved due to the phase-invariance property of the Dysthe equation. The resulting Hamiltonian is given by
\begin{align*}
\widehat \H  =  \H - \varepsilon \omega'_0 I - (\omega_0 - {k}_0 \omega'_0 ) M,
\end{align*}
which, after introducing a new long-time scale $\tau = \varepsilon^2 t$, leads to the following version of the Hamiltonian Dysthe equation 
\begin{equation}  
\label{dysthe-reduced}
\begin{aligned}
i \, \partial_\tau u & = - \frac{1}{2} \omega_0'' \partial_X^2 u + \alpha |u|^2 u + \frac{i}{6} \varepsilon \omega_0''' \partial_X^3 u 
- i \varepsilon \beta |u|^2 \partial_X u - \varepsilon \frac{\gamma k_0^2}{2} u |D_X| |u|^2.
\end{aligned}
\end{equation}

\begin{remark}The Hamiltonian Dysthe equation derived by \cite{CGS21} for two-dimensional deep-water surface gravity waves ($\D = 0$, $\P = 0$) reads 
\begin{equation}
\label{dysthe-surface-gravity}
\begin{aligned}
i \, \partial_\tau u = \frac{g}{8\omega_0^3}  \partial_X^2 u + k_0^3 |u|^2 u + i \varepsilon \frac{g^3}{16 \omega_0^5} \partial_X^3 u 
- 3i \varepsilon k_0^2 |u|^2 \partial_X u - \varepsilon k_0^2 u|D_X| u^2.
\end{aligned}
\end{equation}
It is natural to expect that, in the limit $\D \to 0$ and $\P \to 0$, the coefficients of \eqref{dysthe-reduced} converge to those of \eqref{dysthe-surface-gravity}. 
Indeed, a direct computation verifies this convergence for the coefficients of the linear terms as well as $\alpha \to k_0^3$, $\beta \to 3k_0^2$
together with $\gamma \to 2$ for any $k_0$ because resonant triads are ruled out in this situation, hence $\chi_{\mathcal{N}_\mu} = 0$. 
\end{remark}

\begin{remark}
We have checked that the coefficients $\omega_0''$ and $\alpha$ in the NLS part of \eqref{Dysthe} coincide with those 
of the NLS equation proposed by \cite{TMPV18} for $g$, ${\D} \neq 0$, and by \cite{SS22} for $g$, ${\P} \neq 0$.
\cite{TMPV18} employed the same Cosserat model \eqref{curvature-relation} for ice bending but ignored ice compression.
\cite{SS22} examined ice bending and compression but they considered a non-conservative KL-type model for ice bending.
\end{remark}


\subsection{Envelope solitons}

We take this opportunity to highlight a salient effect of ice compression when added to ice bending.
As analyzed by \cite{A93}, a general result about the NLS equation in the focusing regime is that it admits
permanent envelope solitons (i.e. solitary wavepackets) provided that $c = c_g$ is realized in the wave system.
If so, the wavepacket has carrier wavenumber $k_{\min}$ and propagates in such a way that 
its crests are stationary relative to its envelope.

In the flexural-gravity case (${\P} = 0$), it has been revealed that such envelope solitons do not exist \cite{GP12,MV-BW11},
i.e. the corresponding NLS equation is of defocusing type because
\[
{\rm BFI} = -\omega_0'' \alpha  = -0.006 < 0, \quad \mbox{at} \quad k = k_{\min} = 0.76,
\]
according to the Benjamin--Feir (BF) criterion for modulational instability,
and with $k_{\min}$ defined by \eqref{kmin}.
On the other hand, we get
\[
{\rm BFI} = -\omega_0'' \alpha  = 0.2954 > 0, \quad \mbox{at} \quad k = k_{\min} = 0.88,
\]
for ${\P} = 1$, which implies a focusing regime.

Figure \ref{BFI_nls} portrays BFI for $k_{\min}$ as a function of ${\P}$.
As explained earlier, we restrict ourselves to the interval $0 \le {\P} < 2$ 
where non-conservative exponential wave growth/decay is ruled out.
We find that the change from defocusing to focusing occurs around ${\P} = 0.39$ after which BFI quickly increases.
This regime transition is significant, varying from negative BFI of low magnitude to 
positive BFI of high magnitude over a relatively short interval of ${\P}$.
For e.g. ${\P} = 1.9$, the much larger value of BFI $\simeq 55$ is beyond the range displayed in this figure.

\section{Numerical results} 
\label{section-numerical-results}

We show numerical simulations to illustrate the performance of our Hamiltonian Dysthe equation.
We consider the modulational instability of Stokes waves in the presence of sea ice
and examine the influence of ice compression.
We compare these results to predictions by the cubic NLS model
and to direct simulations of the full Euler system.
We also test the capability of our reconstruction procedure against a more conventional approach.

\subsection{Stability of Stokes waves}

We first give the theoretical prediction for modulational or BF instability of Stokes waves in sea ice.
These are represented in \eqref{Dysthe} by the exact uniform solution 
\begin{equation} \label{steady}
u_0(t)  = B_0 e^{-{\rm i} (\omega_0 + \varepsilon^2 \alpha B_0^2) t},
\end{equation}
where $B_0$ is a positive real constant.
In the gravity water wave case (${\D} = 0$, ${\P} = 0$), such a solution is known to be linearly unstable 
with respect to sideband (i.e. long-wave) perturbations.

The formal calculation consists in linearizing \eqref{Dysthe} about $u_0$ by inserting a perturbation of the form
\[
u(X,t) = u_0(t) \big[ 1 + B(X,t) \big],
\]
where
\[
B(X,t) = B_1 e^{\Omega t + {\rm i} \lambda X} + B_2 e^{\overline \Omega t - {\rm i} \lambda X},
\]
and $B_1, B_2$ are complex coefficients. We find that the condition 
$\operatorname{Re}(\Omega) \neq 0$ for instability implies
\begin{equation} \label{BFCond}
\alpha_1 = -\frac{\omega_0''}{2} \lambda^2 \left[ 
2 B_0^2 \Big( \alpha - \varepsilon \frac{\gamma k_0^2}{2} |\lambda| \Big) + \frac{\omega_0''}{2} \lambda^2 \right] > 0.
\end{equation}
This is a tedious but straightforward calculation for which we omit the details \cite{CGS21,GKS22-jfm}.

Figure \ref{BF_inst} depicts the normalized growth rate 
\[
\frac{|\textrm{Re}(\Omega)|}{\omega_0} = \frac{\sqrt{\alpha_1}}{\omega_0},
\]
delimiting the instability region as predicted by condition \eqref{BFCond}
for $(A_0,k_0) = (0.1,0.9)$ and $(0.01,5)$ with varying values of ${\P}$.
These parameter regimes are representative of the two different cases $\gamma = 2$ and $1$ 
for the mean-flow term, as discussed in Section 6.1. 
They correspond to initial wave steepness $\varepsilon = k_0 A_0 = 0.09$ and $0.05$ respectively.
Note that the envelope amplitude $B_0$ and the surface amplitude $A_0$ are related by
\begin{equation} \label{amplitude}
B_0 = A_0 \sqrt{\frac{\omega_0}{2k_0}},
\end{equation}
according to \eqref{eta-xi-to-z-mapping} and \eqref{z-u-relation-physical}.

We see in figure \ref{BF_inst} that instability tends to be enhanced with increasing ${\P}$.
The growth rate is especially strong for $(A_0.k_0) = (0.1,0.9)$ and ${\P} = 1.9$.
By contrast, the variations in ${\P}$ are weaker for $(A_0,k_0) = (0.01,5)$.
Maximum growth rate occurs around $\lambda = 0.02$ for $(A_0.k_0) = (0.1,0.9)$
and around $\lambda = 0.1$ for $(A_0.k_0) = (0.01,5)$.
The wavenumber at maximum growth rate (as well as the extent of the instability region)
differs by an order of magnitude between these two configurations.

The same remark can be made when comparing these results to predictions on BF instability 
for the gravity water wave problem with similar values of $\varepsilon$ \cite{CGS21,GKS22-jfm}.
Including elasticity leads to narrower regions of instability centered around smaller sideband wavenumbers $\lambda$,
which are smaller by orders of magnitude than their counterparts in the pure gravity case.
This implies that a much longer domain is needed in the present simulations 
to observe this instability with such long-wave modulations. 

\subsection{Reconstruction of the original variables}

At any instant $t$, the surface elevation and velocity potential can be reconstructed from the wave envelope
by inverting the normal form transformation. 
This is accomplished by solving the auxiliary system backward from $s = 0$ to $s = -1$,
with ``initial'' conditions given by the transformed variables
\begin{eqnarray} \label{init_eta}
\eta(x,t) \big|_{s=0} & = & \frac{1}{\sqrt{2}} a^{-1}(D) \Big[ u(x,t) e^{{\rm i} k_0 x} + \overline{u}(x,t) e^{-{\rm i} k_0 x} \Big], \\
\xi(x,t) \big|_{s=0} & = & \frac{1}{{\rm i} \sqrt{2}} a(D) \Big[ u(x,t) e^{{\rm i} k_0 x} - \overline{u}(x,t) e^{-{\rm i} k_0 x} \Big],
\label{init_xi}
\end{eqnarray}
according to \eqref{eta-xi-to-z-mapping} and \eqref{z-u-relation-physical}.
In these expressions, $u$ obeys \eqref{Dysthe} and $a^{-1}(D) = \sqrt{|D|/\omega(D)}$.
The final solution at $s = -1$ represents the original variables $(\eta,\xi)$.
During the evolution in $s$, higher-order harmonic contributions to $(\eta,\xi)$ are produced by the nonlinear terms
in \eqref{Eq:Hamiltonflow-Fourier} starting from the first harmonics \eqref{init_eta}--\eqref{init_xi} 
associated with carrier wavenumber $k_0$.

\subsection{Simulations}

We test this Hamiltonian Dysthe equation against the full equations \eqref{Hamiltonian-formulation-spatial} 
in the context of BF instability of Stokes waves in sea ice.
Following a high-order spectral approach \cite{CS93,GP12}, Eqs. \eqref{Hamiltonian-formulation-spatial} 
are discretized in space by a pseudo-spectral method via the fast Fourier transform (FFT).
Nonlinear products are calculated in the physical space, while Fourier multipliers and spatial derivatives
are evaluated in the Fourier space.
The computational domain is set to $0 \le x < L$ with periodic boundary conditions
and is divided into a regular mesh of $N$ collocation points.
Both functions $\eta$ and $\xi$ are expressed in terms of truncated Fourier series
\begin{equation} \label{truncated}
\begin{pmatrix}
\eta(x_j,t) \\
\xi(x_j,t)
\end{pmatrix} = \sum_{p = -N/2}^{N/2-1} \begin{pmatrix}
\hat \eta_p(t) \\
\hat \xi_p(t)
\end{pmatrix} e^{{\rm i} \kappa_p x_j}, \quad \kappa_p = p \Delta \kappa, \quad x_j = j \Delta x, \quad j = 0, \dots, N-1,
\end{equation}
where $\Delta \kappa = 2\pi/L$ and $\Delta x = L/N$.
Note that $\{ \kappa_p \}$ represents the set of discrete values associated with any wavenumber $k_\ell \in \mathbb{R}$.
The DNO is approximated by its series expansion \eqref{DNOseries}  for which a small number $M$ of terms 
is sufficient to achieve highly accurate results in light of its analyticity properties. 
The value $M = 6$ is selected based on previous extensive tests \cite{XG09}.
Time integration of \eqref{Hamiltonian-formulation-spatial} is carried out in the Fourier space 
so that linear terms can be solved exactly by the integrating factor technique.
The nonlinear terms are integrated in time by using a 4th-order Runge--Kutta scheme with constant step $\Delta t$.
More details can be found in \cite{GP12} and \cite{XG09}, with the exception that Eqs. \eqref{Hamiltonian-formulation-spatial} 
for $\bf{v} = (\eta,\xi)^\top$ can be written as
\[
\partial_t \bf{v} = \bf{L} \bf{v} + \bf{N}(\bf{v}),
\]
where the linear part $\bf{L} \bf{v}$ is defined by
\[
\bf{L} \bf{v} = \begin{pmatrix}
0 & G_0 \\
-g - \D \partial_x^4 - \P \partial_x^2 & 0
\end{pmatrix} \begin{pmatrix}
\eta \\
\xi
\end{pmatrix},
\]
and the nonlinear part ${\bf N}({\bf v}) = ({\bf N}_1,{\bf N}_2)^\top$ is given by
\begin{eqnarray}
{\bf N}_1 & = & G(\eta) - G_0, \\
{\bf N}_2 & = & - \frac{1}{2} (\partial_x \xi)^2
+ \frac{1}{2} \frac{\big[ G(\eta) \xi + (\partial_x \eta) (\partial_x \xi) \big]^2}{1 + (\partial_x \eta)^2} \nonumber \\
& & - \D \Big( \partial_s^2 \kappa - \partial_x^4 \eta + \frac{1}{2} \kappa^3 \Big) - \P (\kappa - \partial_x^2 \eta).
\end{eqnarray}
In the Fourier space, the integrating factor $\Theta_k(t)$ associated with $\bf{L} \bf{v}$ takes the form
\[
\Theta_k(t) = \begin{pmatrix}
\cos(\omega_k t) & \sqrt{\frac{G_0}{g + \D k^4 - \P k^2}} \sin(\omega_k t) \\
-\sqrt{\frac{g + \D k^4 - \P k^2}{G_0}} \sin(\omega_k t) & \cos(\omega_k t) 
\end{pmatrix},
\]
for $k \neq 0$, and
\[
\Theta_0(t) = \begin{pmatrix}
1 & 0 \\
-g t & 1
\end{pmatrix},
\]
for $k = 0$ according to l'H\^opital's rule. The resulting 4th-order Runge--Kutta scheme reads
\begin{eqnarray*}
{\bf f}_1 & = & {\bf N}_k({\bf v}_k^n), \\
{\bf f}_2 & = & \Theta_k \left( -\frac{\Delta t}{2} \right) {\bf N}_k \left[ \Theta_k \left( \frac{\Delta t}{2} \right) 
\left( {\bf v}_k^n + \frac{\Delta t}{2} {\bf f}_1 \right) \right], \\
{\bf f}_3 & = & \Theta_k \left( -\frac{\Delta t}{2} \right) {\bf N}_k \left[ \Theta_k \left( \frac{\Delta t}{2} \right) 
\left( {\bf v}_k^n + \frac{\Delta t}{2} {\bf f}_2 \right) \right], \\
{\bf f}_4 & = & \Theta_k (-\Delta t) \, {\bf N}_k \Big[ \Theta_k(\Delta t) \big( {\bf v}_k^n + \Delta t \, {\bf f}_3 \big) \Big], \\
{\bf v}_k^{n+1} & = & \Theta_k(\Delta t) {\bf v}_k^n + \frac{\Delta t}{6} \Theta_k(\Delta t) 
\Big( {\bf f}_1 + 2 \, {\bf f}_2 + 2 \, {\bf f}_3 + {\bf f}_4 \Big),
\end{eqnarray*}
for the solution ${\bf v}_k = (\eta_k,\xi_k)^\top$ advancing from time $t_n$ to time $t_{n+1} = t_n + \Delta t$.

The same numerical methods are used to solve the envelope equation \eqref{Dysthe},
with the same resolutions in space and time.
On the other hand, for the reconstruction procedure, the Fourier integrals in the auxiliary system are not evaluated by a pseudo-spectral method
with the FFT because they cannot be expressed in terms of convolution integrals, 
unlike the case of surface gravity waves \cite{CGS21,GKS22-siam}.
A reason is the more complicated expression of the linear dispersion relation
in this problem due to ice bending and compression, which does not allow for further simplification 
of the interaction kernels in \eqref{Eq:Hamiltonflow-Fourier}.
Instead, these Fourier integrals are computed by direct quadrature. 
More specifically, they are reduced to one-dimensional integrals in $k_1$ (or $k_2$)
by integrating out the Dirac delta functions before applying the trapezoidal rule.
The interval of integration as well as all Fourier coefficients in the integrands
are confined to the truncated spectrum $-N/2 \le p \le N/2-1$ by virtue of \eqref{truncated}.
The coupled system \eqref{Eq:Hamiltonflow-Fourier} is evolved in $s$ via a 4th-order Runge--Kutta scheme,
with step size $\Delta s = 10 \Delta t$ or $100 \Delta t$.
By construction, Eqs. \eqref{Eq:Hamiltonflow-Fourier} are purely nonlinear (i.e. quadratic in nonlinearity)
and do not contain any stiff linear terms. Accordingly, the value of $\Delta s$ 
may be selected so that $0 < \Delta t \ll \Delta s \ll 1$,
which helps speed up the $s$-marching process to mitigate the computational cost entailed by the quadrature rule.
We have checked that using a smaller value of $\Delta s$ (say $\Delta s = \Delta t$) gives similar results.

While the FFT cannot be exploited for the reconstruction procedure, direct quadrature of the Fourier integrals 
in \eqref{Eq:Hamiltonflow-Fourier} was not found to be a major issue in this two-dimensional setting.
Because this computation is not performed at each instant $t$ (only when data are recorded) 
and because it is performed over a short interval $-1 \le s \le 0$, the associated cost is overall insignificant.
We point out again that Eqs. \eqref{Eq:Hamiltonflow-Fourier} are an exact representation of the normal form transformation
to eliminate non-resonant cubic interactions in this problem.
They are solved numerically in their full form, without resorting to any asymptotic approximation.
Indetermination due to $a_1^2$ or $a_3^2$ being singular at $k_1 = 0$ or $k_3 = 0$ in $P_{123}$ and $Q_{123}$
may be lifted by simply setting the corresponding contributions to zero by virtue of the zero-mass assumption.
Because wavenumbers are discrete in this numerical context and to accommodate floating-point arithmetic, 
the characteristic function $\chi_{\mathcal{N}}(k_1,k_2,k_3)$
in \eqref{K3-fourier-eta-xi} is implemented under relaxed conditions
\[
\chi_{\mathcal{N}}(k_1, k_2, k_3) =
\left\{ \begin{array}{l}
1, \quad \text{if} \quad k_1 + k_2 + k_3 = 0, \quad | \omega_1 - \omega_2 + \omega_3| < \Delta, \\
0, \quad \text{otherwise},
\end{array} \right.
\]
with $0 < \Delta \ll 1$.
We have tested different values of $\Delta$, i.e. $\Delta = \{ 10^{-12}, 10^{-6}, 10^{-2} \}$ and obtained similar results,
which suggests that such a resonant triad does not contribute substantially to the reconstruction process.
This result may be attributed to the spectral discretization which limits the possibilities for resonances, 
or to conditions of the modulational regime being simulated for which quartic resonances are dominant.

Initial conditions of the form
\begin{equation} \label{init_u}
u(x,0) = B_0 \Big[ 1 + 0.1 \cos(\lambda x) \Big],
\end{equation}
are prescribed for \eqref{Dysthe}, where a long-wave perturbation of wavenumber $\lambda \ll k_0$ 
is superimposed to a uniform solution of amplitude $B_0$.
For the full system, the initial conditions $\eta(x,0)$ and $\xi(x,0)$
are reconstructed by solving \eqref{Hamiltonian-formulation-spatial} from transformed initial data 
\eqref{init_eta}--\eqref{init_xi} combined with \eqref{init_u}.
This allows for a meaningful comparison with matching initial conditions.

Figure \ref{L2err_k09_s1_N1024} illustrates the time evolution of the relative $L^2$ errors 
\begin{equation} \label{errors}
\frac{\| \eta_f - \eta_w \|_2}{\| \eta_f \|_2},
\end{equation}
on $\eta$ between the fully ($\eta_f$) and weakly ($\eta_w$) nonlinear solutions for ${\P} = 1$
with $(A_0,k_0,\lambda) = (0.1,0.9,0.02)$ and $(0.01,5,0.1)$ as considered in the previous stability analysis.
The numerical parameters are set to $L = 200\pi$, $N = 1024$ ($\Delta \kappa = 0.01$, $\Delta x = 0.61$), 
$\Delta t = 0.01$ for $(A_0,k_0,\lambda) = (0.1,0.9,0.02)$,
and to $L = 20\pi$, $N = 1024$ ($\Delta \kappa = 0.1$, $\Delta x = 0.06$), 
$\Delta t = 0.0002$ for $(A_0,k_0,\lambda) = (0.01,5,0.1)$.
As mentioned earlier, a sufficiently long domain is specified in $x$ so that long-wave modulations can be resolved.
Needless to say the time scale of these modulations is commensurate with their wavelength.
For reference, errors from the cubic NLS equation (by neglecting the higher-order terms in \eqref{Dysthe}) are also shown in this figure.
The same reconstruction procedure for $\eta$ based on \eqref{Eq:Hamiltonflow-Fourier} is adopted in both cases.
Overall, $L^2$ errors from the NLS equation are noticeably higher than those from the Dysthe equation,
which confirms the superiority of the latter model in this asymptotic regime.
Their values tend to gradually grow over time due to accumulation and amplification of numerical errors during the BF instability process.
For $(A_0,k_0,\lambda) = (0.1,0.9,0.02)$, these two models differ in performance by almost an order of magnitude.
If we switched from $\gamma = 1$ to $\gamma = 2$ for $(A_0,k_0,\lambda) = (0.01,5,0.1)$,
these curves would be flipped with NLS errors being lower than Dysthe errors.
Therefore, figure \ref{L2err_k09_s1_N1024} may serve to validate the formula \eqref{gamma-value} of $\gamma$,
hence our treatment of cubic resonances.

A comparison of surface profiles $\eta$ obtained from these two models and the full nonlinear (Euler) system
is presented in figures \ref{wave_a01_k09_K2_s1_N1024} and \ref{wave_a001_k5_K1_s1_N1024} 
with snapshots at various instants for ${\P} = 1$. 
Again, the two configurations $(A_0,k_0,\lambda) = (0.1,0.9,0.02)$ and $(0.01,5,0.1)$ are examined.
Due to the large aspect ratio and the presence of short oscillations,
we simply plot the local maxima (i.e. crests) and local minima (i.e. troughs) of the Euler solution by using dots,
without showing its entire profile to avoid cluttering these graphs.
Development of the BF instability and associated near-recurrence over time are clearly observed.
The snapshots at $t = 5000$ (figure \ref{wave_a01_k09_K2_s1_N1024}) and at $t = 170$ (figure \ref{wave_a001_k5_K1_s1_N1024}) 
correspond to a time of maximum wave growth with modes $\lambda = 0.02$ 
(two humps $|p| = 2$) and $\lambda = 0.1$ (a single hump $|p| = 1$) being amplified respectively,
which is consistent with the previous stability analysis (see figure \ref{BF_inst}).

Here the waves travel from left to right but, owing to the periodic boundary conditions,
they re-enter the domain from one side whenever they exit it through the other side.
A striking feature of figure \ref{wave_a01_k09_K2_s1_N1024} for $(A_0,k_0,\lambda) = (0.1,0.9,0.02)$ 
is the excellent match in phase, amplitude and shape between the Dysthe and Euler solutions throughout the entire simulation.
By contrast, the NLS wavepacket seems to move faster and the resulting shift accentuates over time.
A slight left-right asymmetry is discernible near the tails of the Dysthe and Euler pulses,
especially when wave modulation is strong, while such an asymmetry is absent from the NLS pulse.
The presence of spatial derivatives of odd order in the Dysthe equation may explained this phenomenon
as opposed to the NLS equation.
For $(A_0,k_0,\lambda) = (0.01,5,0.1)$, while very good agreement is obtained 
between the Dysthe and Euler solutions at early stages of their time evolution, 
discrepancies arise during subsequent cycles of BF instability as shown in figure \ref{wave_a001_k5_K1_s1_N1024}.
Notably, the Dysthe wavepacket is found to be steeper (resp. less steep) than the Euler wavepacket
at $t = 620$ (resp. $t = 980$), although their phases still match well.
In spite of these differences, all three solutions behave qualitatively in the same way
during the process of wave modulation-demodulation.
These results are in line with the $L^2$ errors reported in figure \ref{L2err_k09_s1_N1024}.
Similar observations can be made for other values of $0 \le {\P} < 2$ and are not displayed for convenience.

We further assess the performance of our reconstruction procedure by testing it against independent predictions
from the NLS equation proposed by \cite{TMPV18} for flexural-gravity waves (${\P} = 0$) on deep water.
In this framework, the surface variables are determined perturbatively by a Stokes expansion
\begin{equation} \label{stokes}
\eta = \eta_1 e^{{\rm i} \theta} + \eta_2 e^{2 {\rm i} \theta} + {\rm c.c.}, \quad
Q = \omega_0 \eta_1 e^{{\rm i} \theta} + 2 \omega_0 \eta_2 e^{2 {\rm i} \theta} + {\rm c.c.},
\end{equation}
up to second order, where ${\rm c.c.}$ denotes complex conjugation, $\theta = k_0 x - \omega_0 t$ and $Q = \partial_x \xi$.
The surface velocity potential $\xi$ can be readily deduced from $Q$ as $\xi = \partial_x^{-1} Q$ 
by using the FFT combined with the zero-mass assumption (to handle indetermination at $k = 0$).
The carrier wave envelope $\eta_1$ obeys 
\begin{equation} \label{NLS2}
{\rm i} \, (\partial_t + \omega_0' \partial_x) \eta_1 + \frac{\omega_0''}{2} \partial_x^2 \eta_1 + \Gamma |\eta_1|^2 \eta_1 = 0,
\end{equation}
with $\omega_0^2 = g k_0 + k_0^5$ and 
\[
\omega_0' = \frac{g + 5 k_0^4}{2 \omega_0}, \quad
\omega_0'' = -\frac{g^2 - 30 g k_0^4 - 15 k_0^8}{4 \omega_0^3}, \quad
\Gamma = -\frac{\omega_0 k_0^2 (4 g^2 - 27 g k_0^4 + 44 k_0^8)}{2 (g + k_0^4) (g - 14 k_0^4)},
\]
while the second-harmonic component $\eta_2$ is bound to $\eta_1$ via
\begin{equation} \label{second}
\eta_2 = \frac{\omega_0^2 \eta_1^2}{g - 14 k_0^4}.
\end{equation}
Recall that the coefficients $\omega_0''$ and $\Gamma$ for dispersion and nonlinearity coincide with ours in this situation.
The NLS equation \eqref{NLS2} was derived following the method of multiple scales based on 
an integral reformulation of the Euler system \cite{TMPV18}.
The algebraic equations \eqref{stokes} to reconstruct $(\eta,\xi)$ are more explicit and more efficient 
than \eqref{Eq:Hamiltonflow-Fourier}, but they are perturbative and devoid of any Hamiltonian structure.
On the other hand, our reconstruction procedure is non-perturbative and involves solving 
an auxiliary system of Hamiltonian differential equations, as motivated by the basic Hamiltonian formulation of this problem.

Restricting our attention to ${\P} = 0$, figure \ref{L2nls_k09_s0_N1024} compares $L^2$ errors \eqref{errors} from our NLS equation 
with $(\eta,\xi)$ calculated according to \eqref{Eq:Hamiltonflow-Fourier} and from its counterpart \eqref{NLS2} 
with $(\eta,\xi)$ given by \eqref{stokes}.
For each of these models, the $L^2$ error is estimated relative to the Euler solution with respective initial conditions
$\eta(x,0)$ and $\xi(x,0)$. In the case where \eqref{NLS2} is tested against \eqref{Hamiltonian-formulation-spatial}, 
these initial conditions are provided by \eqref{stokes} and \eqref{second} with
\[
\eta_1(x,0) = \frac{A_0}{2} \Big[ 1 + 0.1 \cos(\lambda x) \Big],
\]
noting again that $A_0$ and $B_0$ are related through \eqref{amplitude} to ensure matching wave parameters
between these two NLS models.
We adopt the same numerical schemes as described earlier, together with the same resolutions in space and time,
to solve \eqref{NLS2} and evaluate \eqref{stokes}.
The initially perturbed Stokes wave is again prescribed by $A_0 = 0.1$, $k_0 = 0.9$ and $\lambda = 0.02$.
We find that the two sets of errors are indistinguishable as shown in this figure.
This is not surprising given the fact that, for such mild initial conditions and for such a weakly nonlinear solution
as produced by the NLS equation, high-order harmonics are not expected to contribute significantly to the wave spectrum.
In light of this remarkable agreement, figure \ref{L2nls_k09_s0_N1024} helps provide a cross-validation
of these two independent methods for reconstructing $(\eta,\xi)$ from the wave envelope:
on one hand, our dynamical system approach based on the numerical solution of \eqref{Eq:Hamiltonflow-Fourier}
and, on the other hand, an asymptotic approach based on the second-order algebraic formulas \eqref{stokes}.
Such a validation was not presented by \cite{TMPV18}.

Finally, the conservation of energy $H$ and wave action $M$ by the Dysthe equation \eqref{Dysthe} is demonstrated
in figure \ref{ener_a01_k09_K2_N1024} where the relative errors 
\[
\frac{\Delta H}{H_0} = \frac{|H - H_0|}{H_0}, \quad \frac{\Delta M}{M_0} = \frac{|M - M_0|}{M_0}, 
\]
are plotted as functions of time for $A_0 = 0.1$, $k_0 = 0.9$ and $\lambda = 0.02$.
Integration over $x$ in \eqref{reduced-H-for-dysthe-new} and in the $L^2$ norm \eqref{errors} 
is carried out via the trapezoidal rule over the periodic domain $[0,L]$.
The reference values $H_0$ and $M_0$ denote the initial values of \eqref{reduced-H-for-dysthe-new} and \eqref{action} at $t = 0$.
Overall, both $H$ and $M$ are very well conserved by \eqref{Dysthe} for various values of ${\P}$.
The gradual loss of accuracy over time is likely attributed to accumulation of numerical errors as enhanced by the BF instability.

\section{Conclusions}

The two-dimensional problem of nonlinear hydroelastic waves propagating along an ice sheet on an ocean of infinite depth
is investigated theoretically and computationally.
Of interest is the asymptotic modulational regime which is applicable to wave groups as commonly observed 
under open-water or ice-covered conditions.
Based on the Hamiltonian formulation for nonlinear potential flow coupled to a thin-plate representation of the ice cover
with nonlinear effects from ice bending and compression,
we propose a Hamiltonian version of the Dysthe equation (a higher-order NLS equation) 
for the slowly varying envelope of weakly nonlinear quasi-monochromatic waves.
This model is derived via a systematic approach by applying ideas from Hamiltonian perturbation theory which involve 
canonical transformations such as reduction to normal form, use of a modulational Ansatz and homogenization over short spatial scales.
Due to the more complicated dispersion relation (as compared to e.g. the pure gravity case),
an additional difficulty here is the presence of resonant triads for which we conduct a detailed examination.

Accordingly, special attention is paid to developing a normal form transformation that eliminates non-resonant triads
while accommodating resonant ones.
For this purpose, the cubic resonances are identified and treated separately in the Fourier integral terms.
This leads to corrections in the normal form transformation as well as in the envelope equation, especially for the mean-flow term.
The reduction to normal form is given by an auxiliary system of integro-differential equations in the Fourier space,
and is also endowed with a Hamiltonian structure.
Its significance is two-fold as it provides a non-perturbative scheme to reconstruct the ice-sheet deformation
from the wave envelope by reverting the evolutionary process.
As a consequence, the entire solution (from the envelope equation to the surface reconstruction) fits within a Hamiltonian framework.

Application to the time evolution of perturbed Stokes waves in sea ice is considered.
Inspecting first the BF instability criterion for the NLS part of this model, 
a regime transition from defocusing to focusing occurs when $\P \simeq 0.39$ at $k_{\min}$
(at the minimum phase speed), which implies the existence of small-amplitude traveling wavepackets.
This theoretical result contrasts with previous reports on the non-existence of such solutions (i.e. a defocusing NLS equation at $k_{\min}$)
in the absence of ice compression ($\P = 0$).
It is supportive of recent observations of persistent wave groups in the Arctic Ocean.
This linear stability analysis is then generalized to the higher-order terms for any carrier wavenumber.
Its predictions are tested against numerical solutions of the Dysthe equation in the time domain.
An assessment in comparison to simulations based on the NLS equation and the full Euler system is also shown.
Overall, very good agreement is found, with evidence confirming that the Dysthe equation performs better than the NLS equation.
Finally, the proposed scheme for surface reconstruction is verified against independent computations 
using a more classical approach via a Stokes expansion.

In the future, it would be of interest to carry out a detailed numerical investigation on solitary wavepackets 
of small to large amplitudes over the range $0 \le \P < 2$, possibly induced by a moving load along the ice sheet.
Another possible extension of this work is to tackle the three-dimensional case where 
a similar Hamiltonian formulation can be adopted \cite{G15}.
Furthermore, it is conceivable that the same splitting method can be applied to dealing with cubic resonances
in the related problem of gravity-capillary waves.

\appendix

\section{Useful identities}
\label{appendix-identities}

Here we provide a list of identities commonly used in this manuscript. 

\subsection{Expansions}

Below we list a number of expansion identities which we rely on during the approximation procedure. We do not give any proofs as the derivations are straightforward. Under the modulational regime \eqref{modulation}, we have the following formulas:
\begin{equation}
\label{expansion-identities}
\begin{aligned}
& |k| = k_0 + \varepsilon \lambda + \calO (\varepsilon^2),\\
& \sgn{k}  = 1, \\
& a(k) = \sqrt{\frac{\omega_0}{k_0}} \left(1 - \frac{\varepsilon \lambda}{4 \omega_0^2} \left( g + \P k_0^2 - 3\D k_0^4 \right) \right) + \calO (\varepsilon^2),\\
& \omega(k) = \omega_0 \left( 1+ \frac{\varepsilon \lambda}{2\omega_0^2} \left( g-3\P k_0^2 + 5\D k_0^4 \right) \right) + \calO (\varepsilon^2),\\
& k_1 - k_3 = \varepsilon (\lambda_1 - \lambda_3) ,\\
& a_{1-3} = a(k_1 - k_3) = \frac{g^{1/4}}{ \varepsilon^{1/4} |\lambda_1 - \lambda_3|^{1/4}}\left( 1- \frac{\P}{4g} \varepsilon^2 |\lambda_1-\lambda_3|^2
\right) + \calO (\varepsilon^{7/4}),\\
& \omega_{1-3} = \omega (k_1-k_3) =  
g^{1/2} \varepsilon^{1/2} |\lambda_1 - \lambda_3|^{1/2} \left( 1 - \frac{\P}{2g} \varepsilon^2 |\lambda_1-\lambda_3|^2 \right)
+ \calO (\varepsilon^{9/2}).
\end{aligned}
\end{equation}
Moreover, we get the following expansions for terms of type $a_{1+2} = a(k_1+k_2)$ and $\omega_{1+2} = \omega (k_1+k_2)$:
\begin{equation}
\label{expansion-identities-1plus2}
\begin{aligned}
& a_{1+2} = \sqrt{\frac{\omega_{2k_0}}{2k_0}} \left( 1 - \frac{\varepsilon (\lambda_1 + \lambda_2)}{4 \omega_{2k_0}^2} (g + 4\P k_0^2 - 48 \D k_0^4) \right) + \calO (\varepsilon^2), \\
& \omega_{1+2} = \omega_{2k_0} \left( 1 + \frac{\varepsilon (\lambda_1 + \lambda_2)}{2 \omega_{2k_0}^2} (g -12 \P k_0^2 + 80 \D k_0^4) \right) + \calO (\varepsilon^2).
\end{aligned}
\end{equation}

\subsection{Integral identities}

\begin{lemma}
\label{lemma-integral-identities-1}
Under the modulational Ansatz \eqref{modulation}, given the indices $j, \ell \in \{1,2,3,4\}$ together with $\alpha, \beta \in \{1,2\}$ and $\mu, \nu \in \{3,4\}$, we have 
\begin{equation*}
\begin{aligned}
& \int \lambda_j z_1 z_2 \overline z_3 \overline z_4 \delta_{1+2-3-4} dk_{1234} = \int \lambda_\ell z_1 z_2 \overline z_3 \overline z_4 \delta_{1+2-3-4} dk_{1234},\\
& \int |\lambda_\alpha-\lambda_\mu| z_1 z_2 \overline z_3 \overline z_4 \delta_{1+2-3-4} dk_{1234} = \int |\lambda_\beta-\lambda_\nu| z_1 z_2 \overline z_3 \overline z_4 \delta_{1+2-3-4} dk_{1234}.
\end{aligned}
\end{equation*}
\end{lemma}

\section{Proof of Proposition \ref{lemma-T1-coeff}}
\label{section-lemma-T1-coeff}
Here we outline the main steps to derive the coefficients in \eqref{T-0}. For simplicity, we introduce notations ${\rm I}$ and ${\rm II}$ to denote the first and second lines of \eqref{H4-in-fourier} respectively, so that $\H^{(4)} = {\rm I} + {\rm II}$. 
We rewrite each of ${\rm I}$ and ${\rm II}$ in terms of $z_k$ and $\overline{z}_{-k}$ via \eqref{eta-xi-to-z-mapping}. 

For ${\rm I}$, we have 
\begin{equation*}
\begin{aligned}
{\rm I} & = \frac{1}{32\pi} \int |k_1| |k_4| \left( |k_1| + |k_4| - 2 |k_3+k_4| \right) \frac{a_1 a_4}{a_2 a_3} (z_1-\bar z_{-1}) (z_2+\bar z_{-2}) (z_3+\bar z_{-3}) (z_4-\bar z_{-4}) \delta_{1234} dk_{1234}, \\
& = \int V_{1234} (z_1-\bar z_{-1}) (z_2+\bar z_{-2}) (z_3+\bar z_{-3}) (z_4-\bar z_{-4}) \delta_{1234} dk_{1234},
\end{aligned}
\end{equation*}
where we have used the definition \eqref{D-1234-defn}. 
Extracting terms of type ``$zz\bar z\bar z$'', we have
$$
\begin{aligned}
{\rm I}_{R} = & \int V_{1234} ( -z_1 z_2 \bar z_{-3} \bar z_{-4} - \bar z_{-1} \bar z_{-2} z_3 z_4
 - z_1 \bar z_{-2}  z_{3} \bar z_{-4} - \bar z_{-1} z_2 \bar z_{-3} z_{4}& + z_1 \bar z_{-2} \bar z_{-3} z_{4} \\
& \qquad \qquad + \bar z_{-1} z_2 z_{3} \bar z_{-4}) \delta_{1234} d{k}_{1234}.
\end{aligned}
$$
This integral  can alternatively be written after index rearrangements as
\begin{equation*}
\begin{aligned}
{\rm I}_{R} = \int (-V_{1234} - V_{4321} - V_{1324} - V_{4231} + V_{1432} + V_{3214}) z_1 z_2 \bar z_{-3} \bar z_{-4} \delta_{1234} d{\rm k}_{1234}.
\end{aligned}
\end{equation*}
After re-indexing $(k_1,k_2,k_3,k_4) \to (k_1, k_2, -k_3, -k_4)$, we get 
\begin{equation*}
\begin{aligned}
{\rm I}_{R} = & \int \Big[ -V_{12(-3)(-4)} - V_{(-4)(-3)21} - V_{1(-3)2(-4)}  - V_{(-4)2(-3)1} \\
& \qquad + V_{1(-4)(-3)2} + V_{(-3)21(-4)} \Big] z_1 z_2 \bar z_{3} \bar z_{4} \delta_{1+2-3-4} d k_{1234},
\end{aligned}
\end{equation*}
with coefficient equal to $T_0^{(1)}$. The computations for ${\rm II}$ are similar and involve $T_0^{(2)}$.

\section{Proof of Proposition \ref{lemma-T-nores-coeff}}
\label{section-lemma-T-nores-coeff}

Here we outline the main steps to derive the coefficients in \eqref{T-nores-1}. 
The terms in $\{K^{(3)}, \H_{\rm NoRes}^{(3)}\}$ associated with $T_{\rm NoRes}$ in \eqref{decompH4} are of the form $zz\bar z \bar z$. To find these terms, we compute the Poisson bracket $\{K^{(3)}, \H_{\rm NoRes}^{(3)}\}$ using the expressions \eqref{K3-fourier-z} for $K^{(3)}$ and $\H_{\rm NoRes}^{(3)} = \H^{(3)} - \H_{\rm Res}^{(3)}$ with $\H_{\rm Res}^{(3)}$ given in \eqref{H3-resonant}. To distinguish between the indices associated to $K^{(3)}$ and $\H^{(3)}$, we use $\{ 1, 2, 3 \}$ for $K^{(3)}$ and $\{ 4, 5, 6 \}$ for $\H^{(3)}$. 

To simplify the computations, we use the decomposition $K^{(3)} = K^{(3)}_0 + K^{(3)}_\chi$ with 
\begin{equation}
\label{K3-0-chi-defn}
\begin{aligned}
& K^{(3)}_0 = ~ \frac{1}{8i\sqrt{\pi}} \int S_{123} \Big[
\frac{z_1 z_2 z_3- \bar z_{-1} \bar z_{-2} \bar z_{-3} }{\omega_1 + \omega_2+ \omega_3} - 
2 \frac{z_1 z_2 \bar z_{-3}- \bar z_{-1} \bar z_{-2} z_3}{\omega_1+ \omega_2 -\omega_3} \Big] \delta_{123} d{k}_{123}, \\
& K^{(3)}_\chi = ~ \frac{1}{8i\sqrt{\pi}} \int B_{123}
\frac{ z_1\bar z_{-2}  z_3-  \bar z_{-1} z_2 \bar z_{-3} }{\omega_1 -\omega_2 +\omega_3}  \delta_{123} d{k}_{123},
\end{aligned}
\end{equation}
where we have used the definition $B_{123} = S_{123} \left( 1-\chi_{123} \right)$. 
Similarly, we decompose $\H^{(3)}_{\rm NoRes} = \H^{(3)}_{{\rm NoRes}, 0} + \H^{(3)}_{{\rm NoRes}, \chi}$ with 
\begin{equation*}
\begin{aligned}
& \H^{(3)}_{{\rm NoRes}, 0} = ~ \frac{1}{8\sqrt{\pi}} \int S_{456} \Big[
z_4 z_5 z_6 + \bar z_{-4} \bar z_{-5} \bar z_{-6} - 
2 (z_4 z_5 \bar z_{-6} + \bar z_{-4} \bar z_{-5} z_6) \Big] \delta_{456} d{k}_{456}, \\
&\H^{(3)}_{{\rm NoRes}, \chi} = ~ \frac{1}{8\sqrt{\pi}} \int B_{456}  
\left( z_4 \bar z_{-5}  z_6 -  \bar z_{-4} z_5 \bar z_{-6} \right)  \delta_{456} d{k}_{456}.
\end{aligned}
\end{equation*}
This allows us to rewrite $\{K^{(3)}, \H^{(3)}_{\rm NoRes}\}$ as the sum of Poisson brackets
\begin{equation}
\label{possion-bracket-decomposition}
\{K^{(3)}_0, \H^{(3)}_{{\rm NoRes}, 0}\} + \{K^{(3)}_0, \H^{(3)}_{{\rm NoRes}, \chi}\} + \{K^{(3)}_\chi, \H^{(3)}_{{\rm NoRes}, 0}\} + \{K^{(3)}_\chi, \H^{(3)}_{{\rm NoRes}, \chi}\}.
\end{equation}
Then, the first bracket in \eqref{possion-bracket-decomposition} implies $T_{\rm NoRes}^{(1)}$, the sum of the two brackets in the middle implies $T_{\rm NoRes}^{(2)}$ and the last bracket gives $T_{\rm NoRes}^{(3)}$. Below we only show the computations for the last bracket in \eqref{possion-bracket-decomposition}
and verify that it leads to $T_{\rm NoRes}^{(3)}$ in \eqref{T-nores-1}. The other Poisson brackets can be treated in a similar way.

Using the formula \eqref{poisson-bracket-formula}, we have 
\begin{equation}
\label{poisson-bracket-chi}
\begin{aligned}
i \, \{K^{(3)}_\chi, \H^{(3)}_{{\rm NoRes}, \chi}\}_R = \frac{1}{64\pi i} \Bigg[ & \int B_{123} B_{456} \frac{1}{\omega_1 - \omega_2 + \omega_3} z_1 z_3 \bar z_{-4} \bar z_{-6} \delta_{25} \delta_{123} \delta_{456} dk_{123456} \\
& - 4 \int B_{123} B_{456} \frac{1}{\omega_1 - \omega_2 + \omega_3} z_2 \bar z_{-3} \bar z_{-5} z_6 \delta_{14} \delta_{123} \delta_{456} dk_{123456}\\
& - 4 \int B_{123} B_{456} \frac{1}{\omega_1 - \omega_2 + \omega_3} \bar z_{-2} z_{3} z_{5} \bar z_{-6} \delta_{14} \delta_{123} \delta_{456} dk_{123456}\\
& + \int B_{123} B_{456} \frac{1}{\omega_1 - \omega_2 + \omega_3} \bar z_{-1} \bar z_{-3} z_{4} z_{6} \delta_{25} \delta_{123} \delta_{456} dk_{123456} \Bigg].
\end{aligned}
\end{equation}
The formula above needs some clarifications, since the direct application of \eqref{poisson-bracket-formula} produces ten different integrals. However, it turns out that some of these integrals are equivalent after index rearrangements as in the following example: 
\begin{equation*}
\begin{aligned}
& \int B_{123} B_{456} \frac{1}{\omega_1 - \omega_2 + \omega_3} \bar z_{-1} z_{2} z_{4} \bar z_{-5} \delta_{14} \delta_{123} \delta_{456} dk_{123456}\\
& = \int B_{123} B_{456} \frac{1}{\omega_1 - \omega_2 + \omega_3} z_2 \bar z_{-3} \bar z_{-5} z_6 \delta_{14} \delta_{123} \delta_{456} dk_{123456},
\end{aligned}
\end{equation*}
after rearranging the indices $(1,2,3) \to (3,2,1)$ and $(4,5,6) \to (6,4,5)$ and using the symmetry $B_{123} = B_{321}$.  
As a result, there are two groups of four equivalent integrals which produce the coefficients $4$ on the right-hand side of \eqref{poisson-bracket-chi}. 

To write the integrals in \eqref{poisson-bracket-chi} in terms of $z_1 z_2 \bar z_{3} \bar z_{4}$ as in \eqref{decompH4}, we use an appropriate modification of indices, which we explain now. For the first integral in \eqref{poisson-bracket-chi}, we start by integrating over $k_2$ and $k_5$. Using $\delta_{123}$ and $\delta_{456}$, this leads to 
\begin{equation*}
\begin{aligned}
& \int B_{123} B_{456} \frac{1}{\omega_1 - \omega_2 + \omega_3} z_1 z_3 \bar z_{-4} \bar z_{-6} \delta_{25} \delta_{123} \delta_{456} dk_{123456}\\
& = \int B_{1(-1-3)3} B_{4(-4-6)6} \frac{1}{\omega_1 - \omega_{1+3} + \omega_3} z_1 z_3 \bar z_{-4} \bar z_{-6} \delta_{1346} dk_{1346}.
\end{aligned}
\end{equation*}
Next we re-index $(k_1,k_3,k_4,k_6) \to (k_1,k_2,-k_{3},-k_4)$ and get
\begin{equation*}
\begin{aligned}
& \int B_{1(-1-3)3} B_{4(-4-6)6} \frac{1}{\omega_1 - \omega_{1+3} + \omega_3} z_1 z_3 \bar z_{-4} \bar z_{-6} \delta_{1346} dk_{1346}\\
& = \int B_{1(-1-2)2} B_{3(-3-4)4} \frac{1}{\omega_1 - \omega_{1+2} + \omega_2} z_1 z_2 \bar z_{3} \bar z_{4} \delta_{1+2-3-4} dk_{1234},
\end{aligned}
\end{equation*}
where we have used the symmetry $B_{(-3)(3+4)(-4)} = B_{(3)(-3-4)(4)}$ and $\delta_{1+2-3-4} = \delta({k_1+k_2-k_3-k_4})$. The resulting coefficient
\begin{equation*}
B_{1(-1-2)2} B_{3(-3-4)4} \frac{1}{\omega_1 - \omega_{1+2} + \omega_2},
\end{equation*}
corresponds to the first term on the first line of the expression for $T_{\rm NoRes}^{(3)}$ in \eqref{T-nores-3}. The other integrals in \eqref{poisson-bracket-chi} are treated in a similar way.

\section*{Acknowledgement}
{Part of this work was carried out when A.K. was a Postdoctoral Fellow at the University of Toronto; 
A.K. was partially supported by McMaster University, the University of Toronto and Nazarbayev University under Social Policy grant for 2024-2025. 
P.G. is partially supported by the National Science Foundation (USA) under grant DMS-2307712.
C.S. is partially supported by the Natural Sciences and Engineering Research Council of Canada (NSERC) under grant RGPIN-2024-03886.}




\clearpage

\begin{figure}
\centering
\subfloat[]{\includegraphics[width=.336\linewidth]{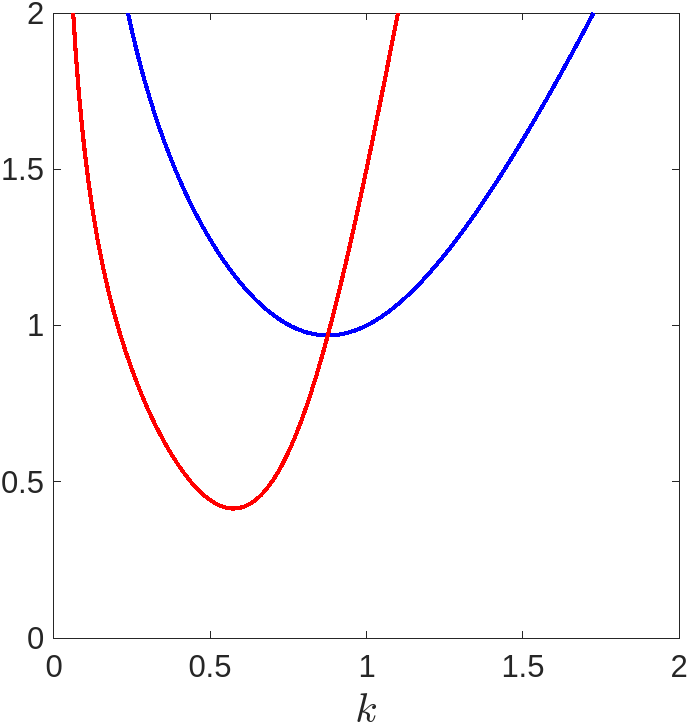}}
\hfill
\subfloat[]{\includegraphics[width=.33\linewidth]{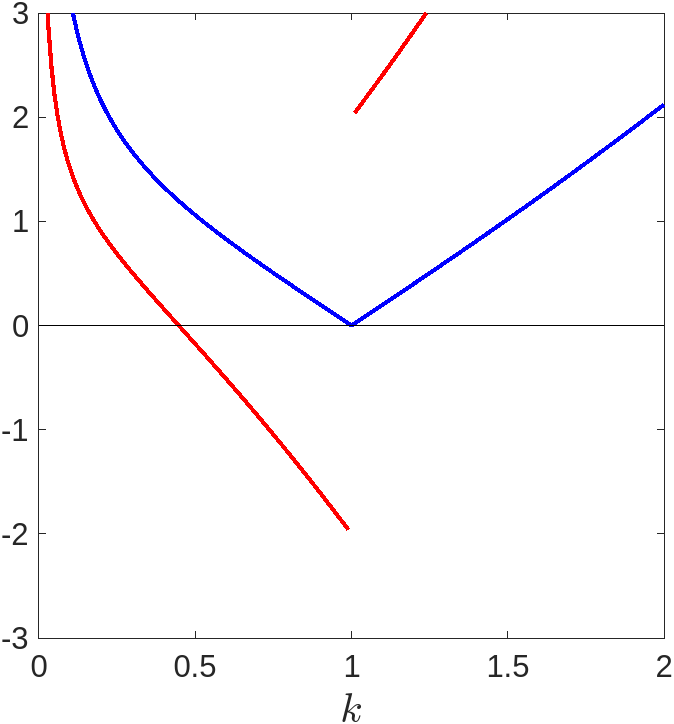}}
\hfill
\subfloat[]{\includegraphics[width=.33\linewidth]{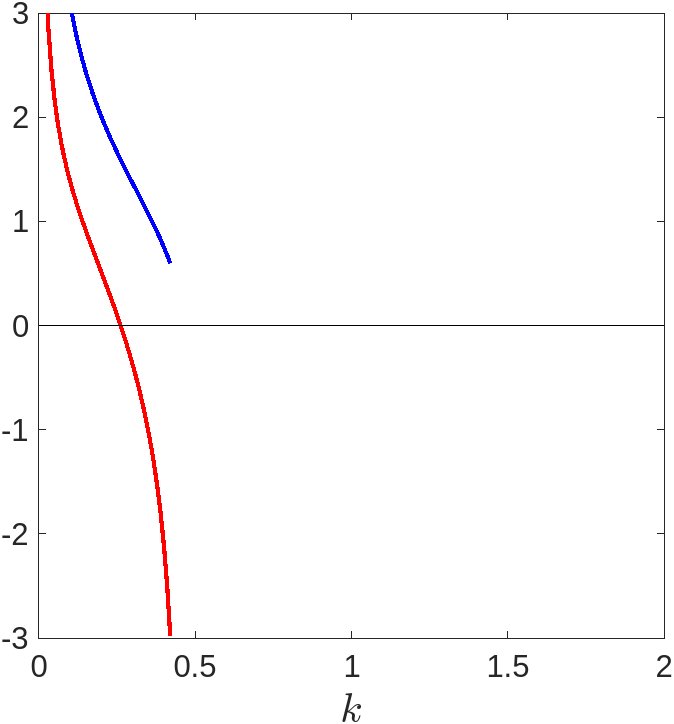}}
\caption{Phase speed $c(k)$ (blue) and group speed $c_g(k)$ (red) as functions of $k$ 
for (a) $\P = 1$, (b) $\P = 2$, (c) $\P = 5$.}
\label{coeff-graph}
\end{figure}

\begin{figure}
\centering
\includegraphics[width=.5\linewidth]{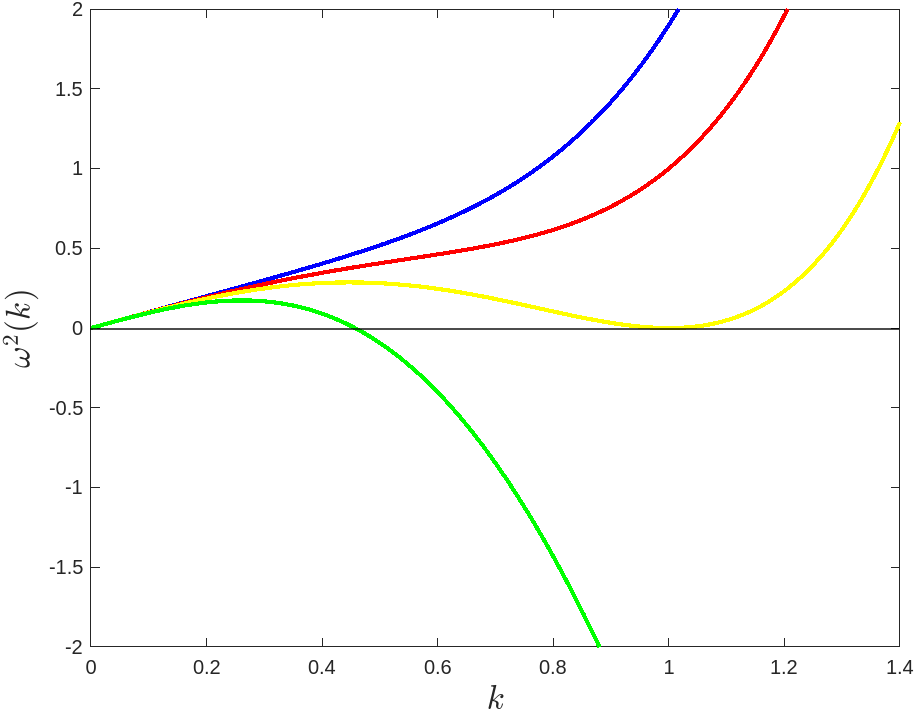}
\caption{Linear dispersion relation $\omega^2(k)$ as a function of $k$ 
for $\P = 0.1$ (blue), $\P = 1$ (red), $\P = 2$ (yellow), $\P = 5$ (green).}
\label{figure-omega2-k}
\end{figure}

\clearpage

\begin{figure}
\centering
\includegraphics[width=.4\linewidth]{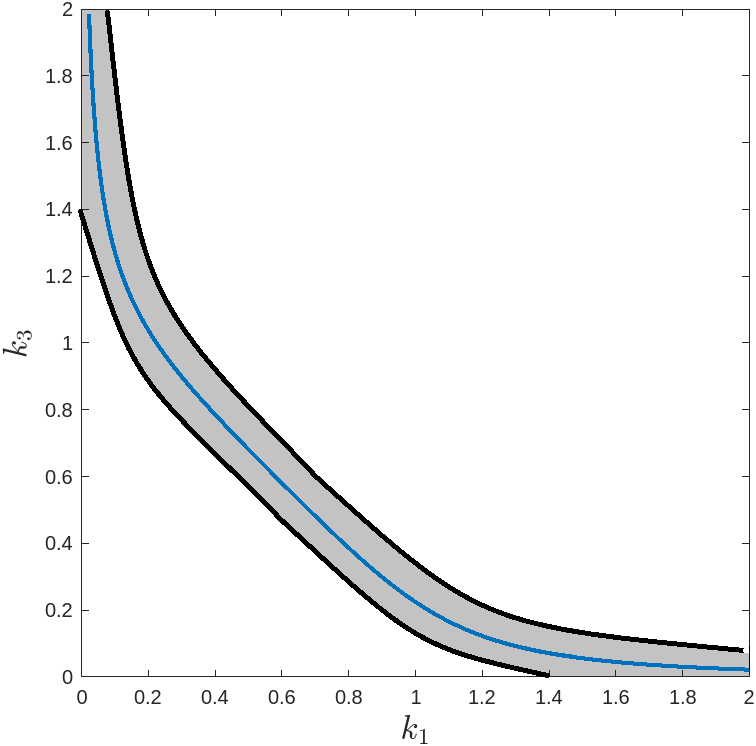}
\caption{Solution curve $\mathcal{C}^+$ (blue curve) for positive roots $(k_1, k_3)$ of \eqref{tilde-d13=0}
and its neighborhood $\mathcal{C}_\mu^+$ (gray area) for $\P = 1$.}
\label{figure-resonant-triads}
\end{figure}

\begin{figure}
\centering
\subfloat[]{\includegraphics[width=.45\linewidth]{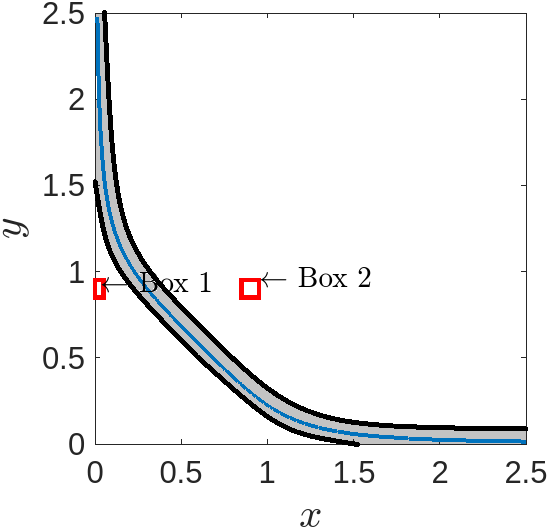}}
\hfill
\subfloat[]{\includegraphics[width=.45\linewidth]{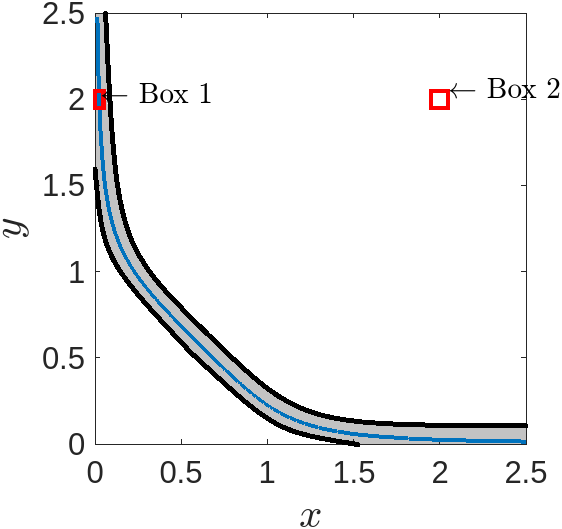}}
\caption{Location of points $(k_1, k_2)$ and $(k_4-k_2, k_2)$ with $k_4-k_2 >0$ 
in the case (a) $k_0 = 0.9$ and (b) $k_0=2$ relative to the neighborhood $\mathcal{C}_\mu$. 
Box 1 represents the set $\mathcal{B}_s (k_0)$ and Box 2 represents the set $\mathcal{B}_c (k_0)$.}
\label{fig:assumptions}
\end{figure}

\clearpage

\begin{figure}
\centering
\includegraphics[width=.5\linewidth]{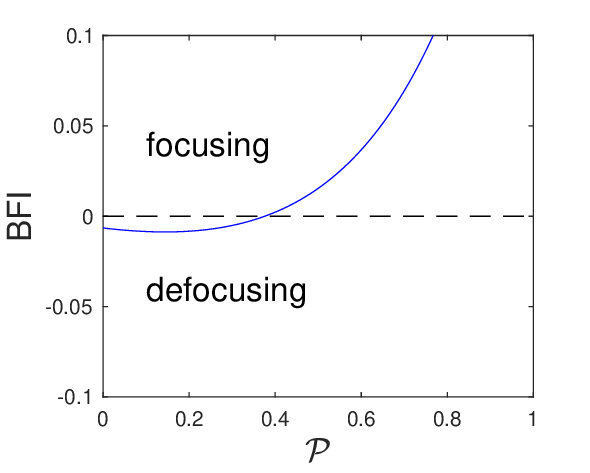}
\caption{BF instability/stability criterion at $k_{\min}$ for the NLS equation as a function of $\P$.}
\label{BFI_nls}
\end{figure}

\clearpage

\begin{figure}
\centering
\subfloat[]{\includegraphics[width=.5\linewidth]{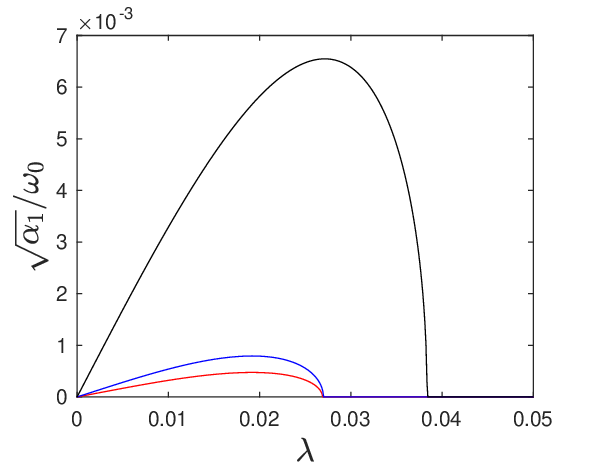}}
\hfill
\subfloat[]{\includegraphics[width=.5\linewidth]{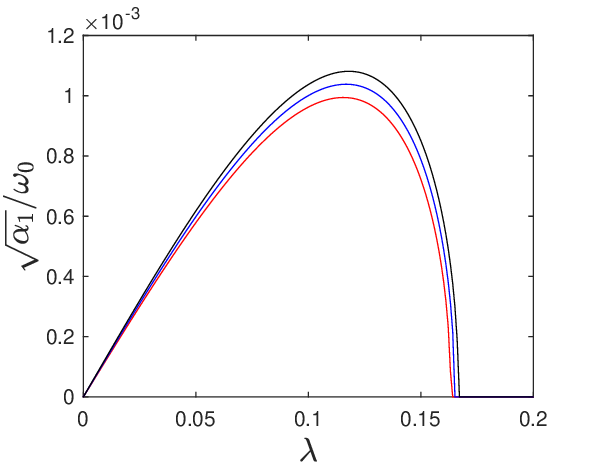}}
\caption{Regions of BF instability according to \eqref{BFCond} for
(a) $(A_0,k_0) = (0.1,0.9)$ and (b) $(A_0,k_0) = (0.01,5)$.
The various curves represent $\P = 0$ (red), $\P = 1$ (blue), $\P = 1.9$ (black).}
\label{BF_inst}
\end{figure}

\begin{figure}
\centering
\subfloat[]{\includegraphics[width=.5\linewidth]{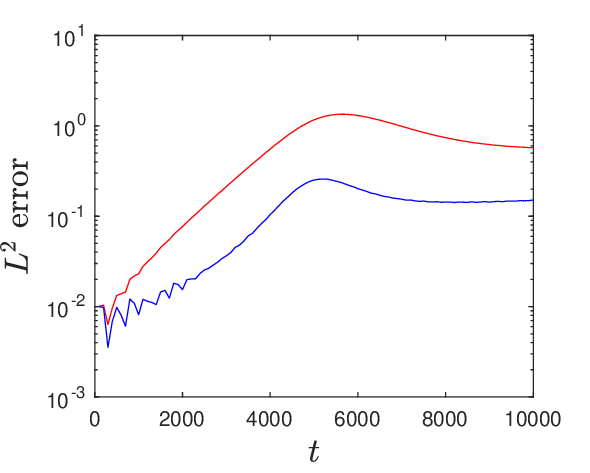}}
\hfill
\subfloat[]{\includegraphics[width=.5\linewidth]{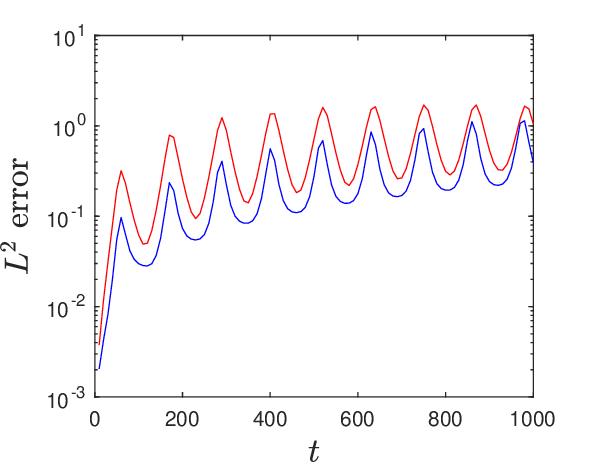}}
\caption{Relative errors on $\eta$ between fully and weakly nonlinear solutions for $\P = 1$
with (a) $(A_0,k_0,\lambda) = (0.1,0.9,0.02)$ and (b) $(A_0,k_0,\lambda) = (0.01,5,0.1)$.
Blue curve: Dysthe equation.
Red curve: NLS equation.}
\label{L2err_k09_s1_N1024}
\end{figure}

\clearpage

\begin{figure}
\centering
\subfloat[]{\includegraphics[width=.5\linewidth]{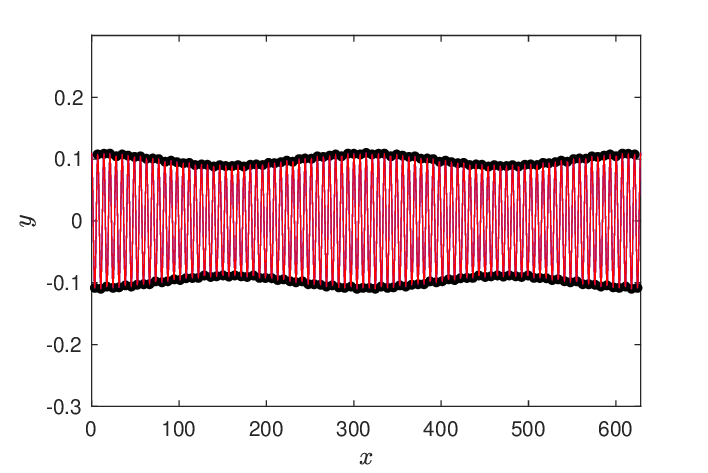}}
\hfill
\subfloat[]{\includegraphics[width=.5\linewidth]{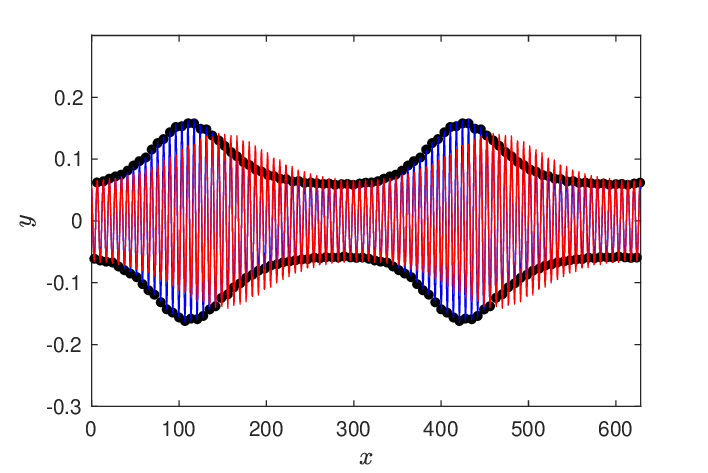}}
\hfill
\subfloat[]{\includegraphics[width=.5\linewidth]{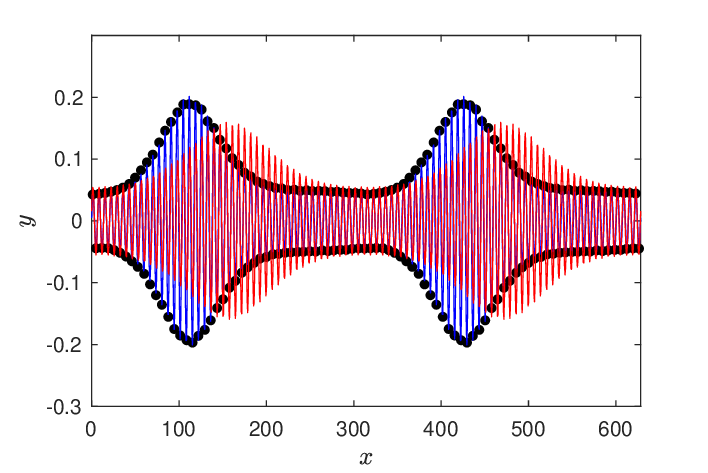}}
\hfill
\subfloat[]{\includegraphics[width=.5\linewidth]{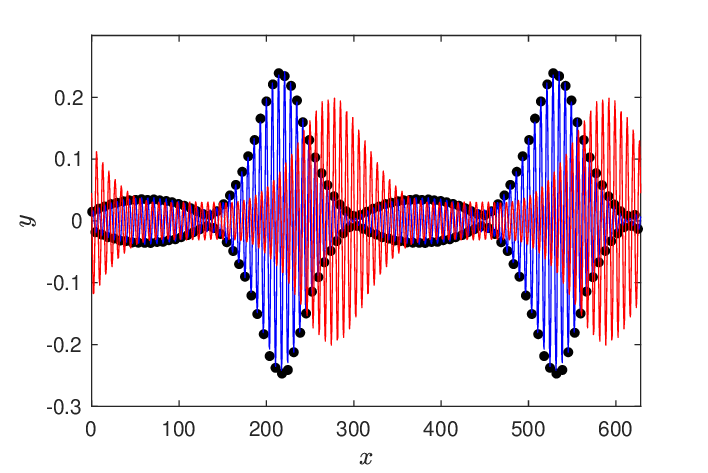}}
\hfill
\subfloat[]{\includegraphics[width=.5\linewidth]{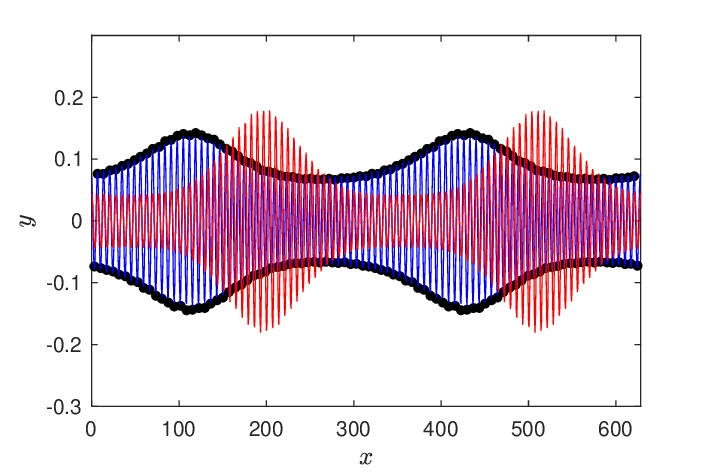}}
\hfill
\subfloat[]{\includegraphics[width=.5\linewidth]{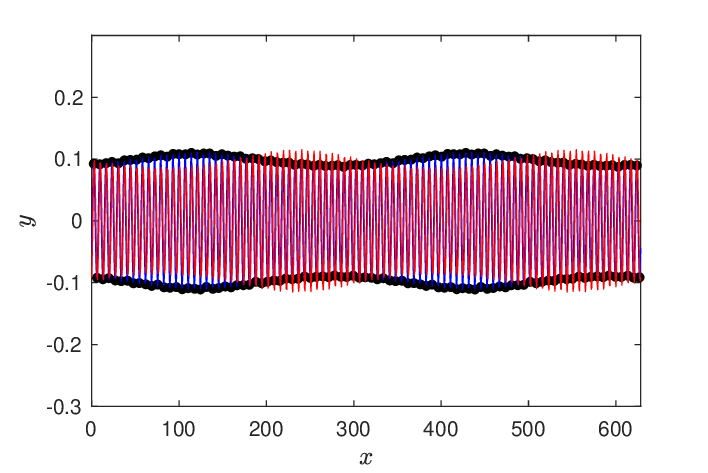}}
\caption{Comparison on $\eta$ between fully and weakly nonlinear solutions
for $(A_0,k_0,\lambda) = (0.1,0.9,0.02)$ and $\P = 1$ at 
(a) $t = 0$, (b) $t = 3400$, (c) $t = 4000$, (d) $t = 5000$, (e) $t = 7000$, (f) $t = 10000$.
Blue curve: Dysthe equation.
Red curve: NLS equation.
Black dots: Euler system.}
\label{wave_a01_k09_K2_s1_N1024}
\end{figure}

\clearpage

\begin{figure}
\centering
\subfloat[]{\includegraphics[width=.5\linewidth]{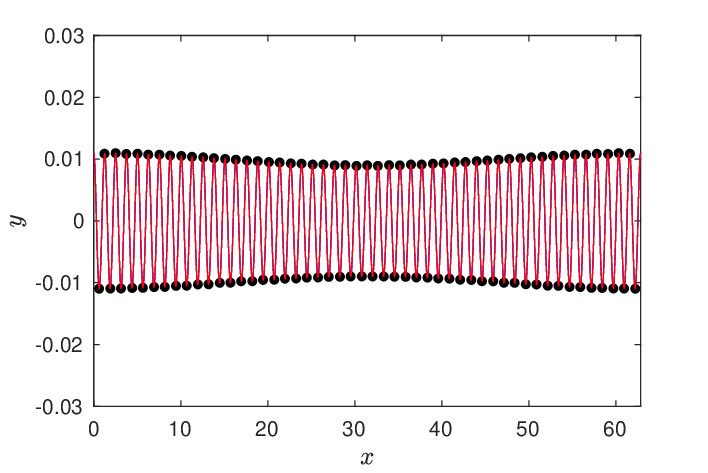}}
\hfill
\subfloat[]{\includegraphics[width=.5\linewidth]{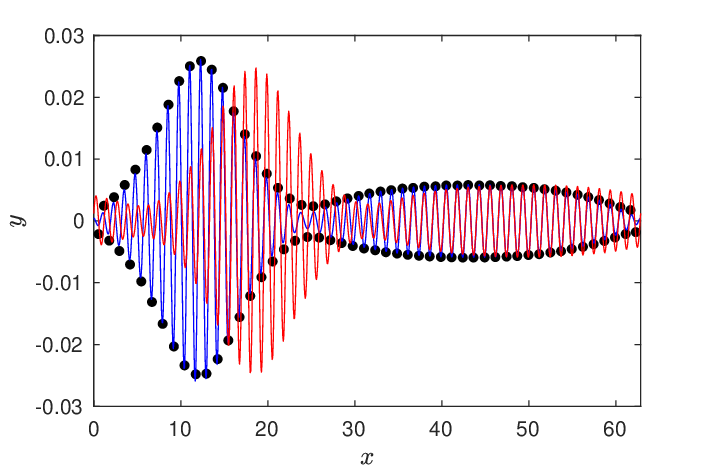}}
\hfill
\subfloat[]{\includegraphics[width=.5\linewidth]{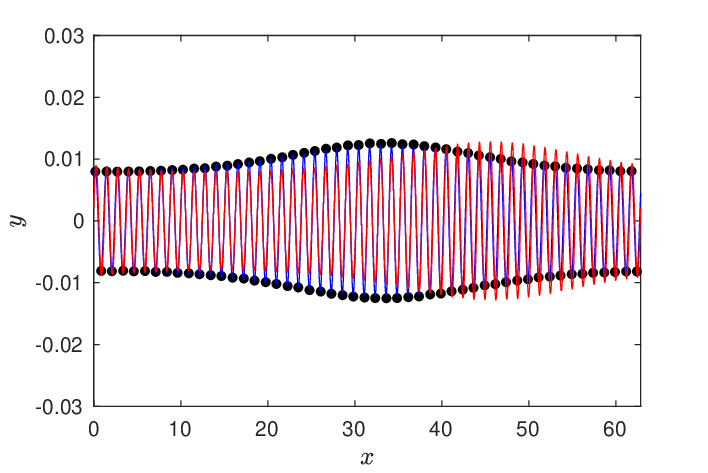}}
\hfill
\subfloat[]{\includegraphics[width=.5\linewidth]{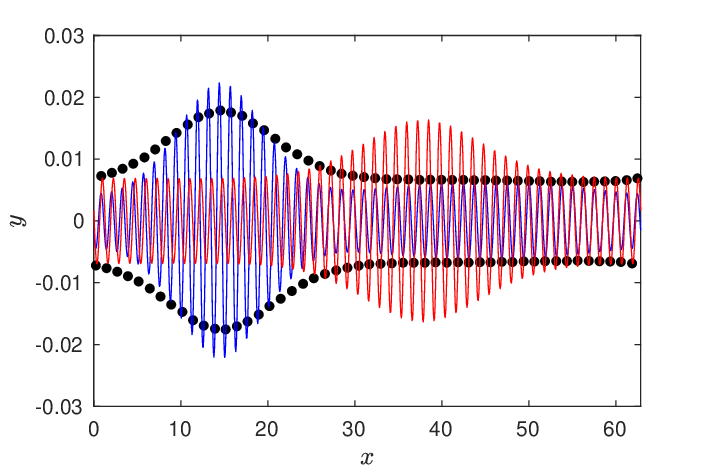}}
\hfill
\subfloat[]{\includegraphics[width=.5\linewidth]{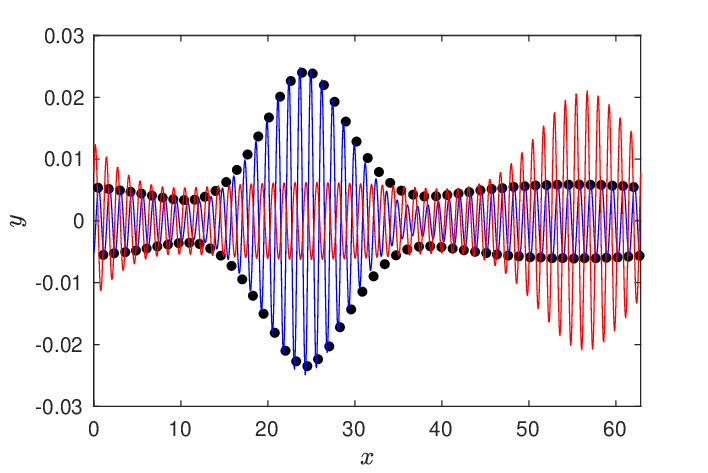}}
\hfill
\subfloat[]{\includegraphics[width=.5\linewidth]{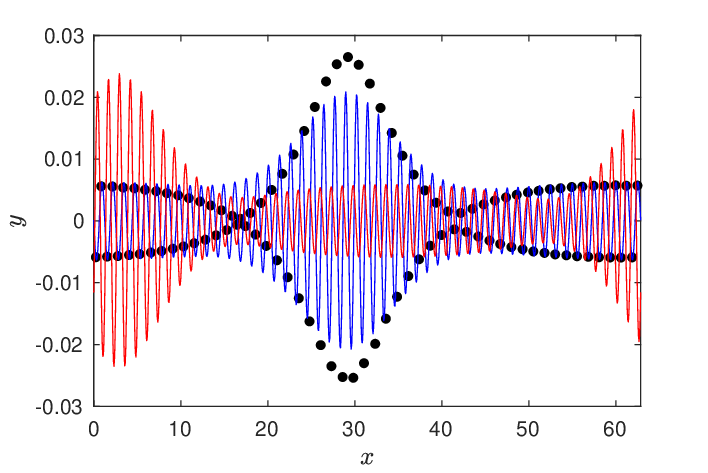}}
\caption{Comparison on $\eta$ between fully and weakly nonlinear solutions
for $(A_0,k_0,\lambda) = (0.01,5,0.1)$ and $\P = 1$ at 
(a) $t = 0$, (b) $t = 170$, (c) $t = 320$, (d) $t = 620$, (e) $t = 860$, (f) $t = 980$.
Blue curve: Dysthe equation.
Red curve: NLS equation.
Black dots: Euler system.}
\label{wave_a001_k5_K1_s1_N1024}
\end{figure}

\clearpage

\begin{figure}
\centering
\includegraphics[width=.5\linewidth]{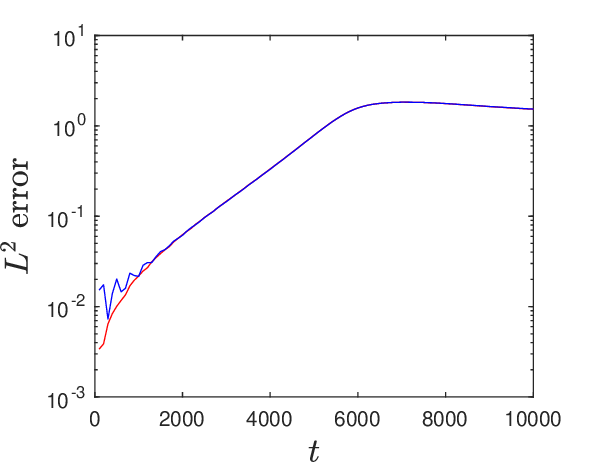}
\caption{Relative errors on $\eta$ between fully and weakly nonlinear solutions for $\P = 0$
with $(A_0,k_0,\lambda) = (0.1,0.9,0.02)$.
Blue curve: our NLS equation.
Red curve: NLS equation from \cite{TMPV18}.}
\label{L2nls_k09_s0_N1024}
\end{figure}

\begin{figure}
\centering
\subfloat[]{\includegraphics[width=.5\linewidth]{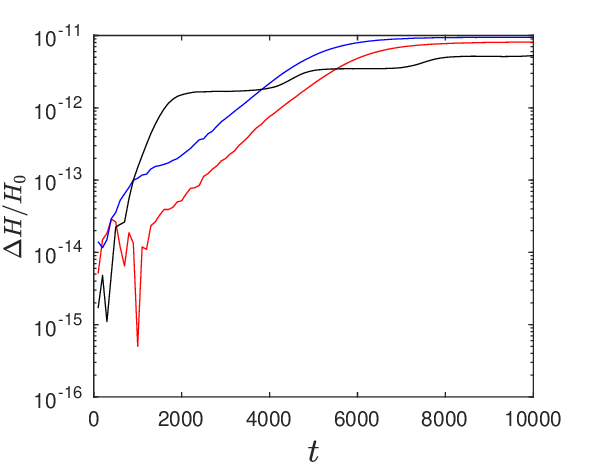}}
\hfill
\subfloat[]{\includegraphics[width=.5\linewidth]{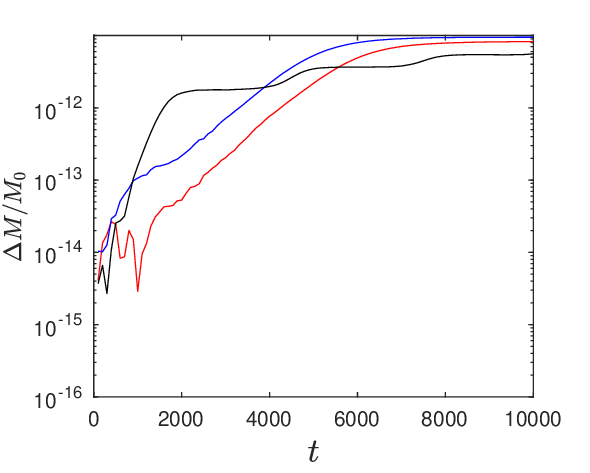}}
\caption{Relative errors on (a) $H$ and (b) $M$ for the Dysthe equation with $(A_0,k_0,\lambda) = (0.1,0.9,0.02)$
and $\P  = 0$ (red), $\P = 1$ (blue), $\P = 1.9$ (black).}
\label{ener_a01_k09_K2_N1024}
\end{figure}

\end{document}